\documentclass[11pt]{article}

\usepackage[utf8]{inputenc}

\usepackage{amsthm,amsmath,amssymb}
\RequirePackage[numbers]{natbib}
\RequirePackage[colorlinks,citecolor=blue,urlcolor=blue]{hyperref}
\usepackage{array}
\usepackage{algorithm}
\usepackage{algorithmic}
\usepackage{color}
\usepackage{times}
\usepackage{graphicx}

\newtheorem{theorem}{Theorem}
\newtheorem{corollary}{Corollary}
\newtheorem{lemma}{Lemma}

\newtheorem{remark}{Remark}

\def\bo{\mathring{\beta}}
\def\b*{\mathring{\beta}}
\def\Rp{\mathbb{R}^p}
\def\R{\mathbb{R}}
\def\dt{\delta_{t-1}}

\def\bmin{\mathring{\beta}_{min}}

\def\hbj{\hat{\beta} _J}
\def\tbj{\tilde{\beta} _J}
\def\Ex{\mathbb E}
\def\ml{\mathring{\ell}}
\def\Ms{r^*}
\def\cc{\kappa_{a_1}}
\def\hbeta{\hat{\beta}}

\title{Improving Lasso for Model Selection and Prediction}

\author{Piotr Pokarowski \footnote{Institute of Applied Mathematics and Mechanics,
University of Warsaw, Banacha 2, 02-097 Warsaw, Poland, pokar@mimuw.edu.pl. The research was supported by the Polish
National Science Center grant no. 2015/17/B/ST6/01878.}
,Wojciech Rejchel 
\footnote{Faculty of Mathematics and Computer Science,
Nicolaus Copernicus University,  Chopina 12/18, 87-100, Toru\'n, Poland, 
wrejchel@gmail.com. The research was supported by the Polish
National Science Center grant no. 2015/17/B/ST6/01878} 
, Agnieszka So\l tys \footnote{Institute of Applied Mathematics and Mechanics,
University of Warsaw, Banacha 2, 02-097 Warsaw, Poland, a.prochenka@phd.ipipan.waw.pl}, \\
Micha\l \: Frej\footnote{Institute of Applied Mathematics and Mechanics,
University of Warsaw, Banacha 2, 02-097 Warsaw, Poland, m.frej@mimuw.edu.pl}, 
Jan Mielniczuk \footnote{Institute of Computer Sciences, Polish Academy of Sciences, Jana Kazimierza 5, 01-248 Warsaw, Poland, miel@ipipan.waw.pl}}

\date{}

\begin{document}

\maketitle

%\begin{abstract}
%The Lasso, that is the $l_1$-penalized loss estimator, is a popular tool for fitting sparse models to high-dimensional data. 
%The concave regularizations SCAD or MCP approximate more closely the $l_0$-penalized loss, that is the Generalized Information Criterion (GIC), 
%and correct intrinsic estimation bias of the Lasso. 
%In this paper we propose an alternative method of improving the Lasso for predictive models with general convex loss functions which encompass normal linear models, 
%logistic regression, quantile regression or support vector machines. For a given penalty we order the absolute values of the Lasso non-zero coefficients 
%and then select the final model from a small nested family by GIC. We derive an upper bound on the  selection error of the method and show in numerical experiments 
%on synthetic and real-world data sets that an implementation of our algorithm is more accurate than implementations of studied concave regularizations.
%\end{abstract}

\begin{abstract}
It is known that the Thresholded Lasso (TL), SCAD or MCP correct intrinsic estimation bias of
the Lasso. In this paper we propose an alternative method of improving the Lasso for
predictive models with general convex loss functions which encompass normal linear
models, logistic regression, quantile regression or support vector machines. For a given
penalty we order the absolute values of the Lasso non-zero coefficients and then select the
final model from a small nested family by the Generalized Information Criterion.
We derive exponential upper bounds on the selection error of the method. These results
confirm that, at least for normal linear models, our algorithm seems to be the benchmark for
the theory of model selection as it is constructive, computationally efficient and leads to
consistent model selection under weak assumptions. Constructivity of the algorithm means
that, in contrast to the TL, SCAD or MCP, consistent selection does not rely on the unknown
parameters as the cone invertibility factor. Instead, our algorithm only needs the sample
size, the number of predictors and an upper bound on the noise parameter.
We show in numerical experiments on synthetic and real-world data sets that an
implementation of our algorithm is more accurate than implementations of studied concave
regularizations. Our procedure is contained in the R package ,,DMRnet'' and available on the
CRAN repository.
\end{abstract}

List of keywords: convex loss function, empirical process, generalized information criterion, high-dimensional regression,  penalized estimation, selection consistency

\section{Introduction}
Sparse high-dimensional predictive models, where the number of true predictors $t$ is significantly smaller than the sample size $n$ 
and the number of all predictors $p$ greatly exceeds $n,$ have been a focus of research in statistical machine learning in recent years. 
The Lasso algorithm, that is the minimum loss method regularized by sparsity inducing the $\ell_1$ penalty, is the main tool of fitting 
such models \citep{Tibshirani11,BuhlmannGeer11}. However,  it has been shown that the model selected by the Lasso is usually 
too large  and  that for asymptotically consistent model selection it requires the {\it irrepresentable condition} on an experimental matrix 
\citep{MeinshausenBuhlmann06,ZhaoYu06,ShenEtAl12} which is too restrictive in general. The model's dimension can be reduced without loss of the quality 
using the Thresholded Lasso (TL) algorithm, which selects variables with largest absolute values of the Lasso coefficients \citep{YeZhang10,Zhou09} 
or by solving  more computationally demanding minimization of a loss with a folded concave penalty (FCP) as SCAD 
\citep{FanLi01}, MCP \citep{ZhangCH10} or capped $l_1$-penalty \citep{ZhangT10,ShenEtAl12}.  TL, FCP and similar methods lead to consistent selection 
under weaker assumptions such as the {\it restricted isometry property} 
%which means that singular values of normalized experimental submatrices  corresponding to small sets of predictors are uniformly bounded away from zero and infinity 
\citep{ZhangCH10,ZhangT10,ZhangZhang12,ShenEtAl12,WangEtAl13,FanEtAl14,WangEtAl14}.  In \cite{PokarowskiMielniczuk15} one introduced an algorithm 
called {\it Screening--Ordering--Selection} (SOS)  for linear model selection, which reduces successfully the model selected by the Lasso. 
SOS is based on the variant of TL proposed by \cite{Zhou09} and leads to consistent model selection under assumptions similar to the restricted isometry property.

In the paper we consider two algorithms that improve the Lasso, i.e. they are model selection consistent under weaker assumptions than the Lasso. 
Moreover, the considered procedures are computationally simpler than FCP methods that use non-convex penalties. The first algorithm is the well-known Thresholded Lasso. 
Its model selection consistency in the normal linear model is proven in \cite[Theorem 8]{YeZhang10} provided that conditions, which seem to be ,,minimal'', 
are satisfied. Our first contribution is an extension of this result (given in Theorem \ref{th1}) to Generalized Linear Models (GLM).  
However, the TL algorithm is not constructive, because it is not known how to choose the threshold. 
Therefore, in the current paper we propose the second improvement of the Lasso which is the {\it Screening--Selection} (SS) algorithm. 
It is a two-step procedure: in the first step (screening) one computes the Lasso estimator 
$\hat{\beta}$ with penalty $\lambda$ and orders its nonzero coefficients according to their decreasing absolute values. 
In the second step (selection)  one chooses the model which minimizes the Generalized Information Criterion  (GIC)  with penalty $\lambda^2/2$ in a nested family induced 
by the ordering.  Thus, the SS algorithm (Algorithm~\ref{alg:SS} below) can be viewed as the Lasso with adaptive thresholding based on GIC. 
We prove that this procedure is model selection consistent in a wide class of models containing linear models with the subgaussian noise (Theorem \ref{th_LM}), 
GLM (Theorem \ref{th2}) and models with convex (possibly nondifferentiable) loss functions (Theorem \ref{main_emp})  as in quantile regression or support vector machines. 
The obtained results are   exponential upper bounds 
on the selection error of SS in terms of $\lambda$, which parallel  the known bounds for TL  \citep[Theorem~8]{YeZhang10}  
or FCP \citep[Corollary~3 and 5]{FanEtAl14}. For GLM our results are obtained on the basis of exponential inequalities for subgaussian random variables.  
In the case of predictive models with  general convex loss functions we use methods from the empirical process theory. 

The SS algorithm is a simplification and a generalization of the {\it Screening--Ordering--Selection} (SOS) algorithm from \cite{PokarowskiMielniczuk15}
that was proposed only for normal linear models. We show that the ordering step in SOS can be done using  separability of the Lasso instead of using  $t$-statistics. 
Separability means that, under mild assumptions, coordinates of the Lasso corresponding to relevant predictors are larger in absolute values than those corresponding 
to irrelevant predictors. Moreover, we establish that the SS algorithm is model selection consistent beyond normal linear models. 
The new procedure can be applied to the various statistical predictive models including models with quantitative or qualitative response variables. 
We can use ''classical'' models (the normal linear model, the logistic model) as well as modern problems involving piecewise-linear loss functions 
(as in quantile regression or support vector machines).    

Our results state that, in contrast to  TL and FCP methods, SS is constructive in the  linear model  with the subgaussian noise, because 
 it does not rely  on the unknown parameters as the true vector 
$\mathring{\beta}$ or the cone invertibility factors. 
Indeed, we will establish in Section~\ref{sec:models}
that the choice of the tuning parameter $\lambda$ as
\begin{equation}
\label{lambda_n}
\lambda=\sqrt{\frac{2 \sigma ^2 \log p}{n}} (1+o(1))
\end{equation}
leads to model selection consistency of the SS algorithm. 
Therefore, $\lambda$ only depends on $n$, $p$ and $\sigma^2$ that is an upper bound on the noise parameter. 
The assumption that $\sigma^2$  is known  is common in the literature investigating theoretical properties of variable selection procedures \cite{YeZhang10, BuhlmannGeer11,FanEtAl14}.

Moreover, we will also prove that model selection consistency of the SS algorithm holds under weaker conditions than for competitive procedures. For instance, the algorithm of \cite{WangEtAl13} requires that  
\begin{equation}
\label{wang_n}
\frac{\log p}{\delta} =o(1) \quad {\rm and } \quad \frac{t }{\bmin} \sqrt{\frac{\log p}{n}}=o(1),
\end{equation}
 where $\delta$ is a scaled Kullback-Leibler distance between the true set and its submodels (defined in \eqref{delta}) and $\bmin$ is the minimal signal strength. However, for the SS procedure we need that
\begin{equation}
\label{us_n}
\frac{\log p}{\delta} =O(1) \quad {\rm and } \quad \frac{1}{\bmin}\sqrt{\frac{\log p}{n}}=O(1),
\end{equation}
which, among others,  weakens significantly the {\it beta-min} condition. More detailed description of the SS algorithm and comparisons to other procedures is given in Section \ref{sec:models}. This analysis enables us to claim that for subgaussian linear models the SS algorithm seems to be the benchmark for the theory of model selection as it is constructive,
computationally efficient and leads to consistent model selection under weak assumptions. 
Similarly to TL or FCP, the SS algorithm becomes non-constructive, if we go  beyond the  linear model with the subgaussian noise.
Notice that in \eqref{lambda_n}, \eqref{wang_n} and \eqref{us_n} we assume, similarly to \cite{WangEtAl13}, that predictors are scaled to have the $l_2$-norm equal to $\sqrt{n}.$ However, in the rest of the paper we will use more convenient notation that the $l_2$-norm of predictors is one. Therefore, the expression $\sqrt{n}$ will not appear in analogs of \eqref{lambda_n}, \eqref{wang_n} and \eqref{us_n} in Section \ref{sec:models}.

Although TL, FCP or SS algorithms use the Lasso estimators only for one value of the penalty, which is convenient  for theoretical analysis, 
the practical Lasso implementations return coefficient estimators for all possible penalty values as in   the R~package {\tt LARS} 
described in \citet{EfronEtAl04} or for a given net of them as in the R~package {\tt glmnet} described in \citet{FriedmanEtAl10}. 
Similarly, using a net of penalty values, the FCP algorithm has been implemented for linear models in the R~package {\tt SparseNet} \citep{MazumderEtAl11}) 
and for logistic models in  the R~package
{\tt ncvreg} \citep{ncvreg}. 
Our contribution is also  the SSnet algorithm (Algorithm \ref{alg:SOSnet} below), 
which is a {\it practical} version of the SS algorithm. 
 SSnet uses {\tt glmnet} to calculate the Lasso for a net of penalty values, 
then again it selects the final model from a small family by minimizing GIC. In numerical experiments we investigate properties of SSnet in model selection as well as prediction and compare them to competitive procedures.
Using synthetic and real-world data sets we show that SSnet is more accurate than implementations of FCP. The variant of SSnet 
 is contained  in the R package ,,DMRnet'' and available on the CRAN repository.

GIC is a popular procedure in choosing the final model in variable selection. In the literature there are many papers     investigating its properties, 
for instance \cite{KimKwonChoi2012, WangEtAl13} in linear models, \cite{FanTang2013, HuiWartonFoster2015} in GLM, \cite{KATAYAMA2014138} 
in multivariate linear regression,  \cite{ZhangSVM2016} for SVM and \cite{KimJeon2016} for general convex loss functions. 
 GIC is often applied to pathwise algorithms under a common assumption that the true model is on this path. 
This condition excludes  Lasso-based pathwise algorithms as it is only fulfilled under   restrictive assumptions. 
One can overcome this problem using a three-step procedures, for instance the Lasso and non-convex penalized regression \citep{WangEtAl13}  or 
the Lasso and thresholding \citep{KimJeon2016}. 
However, it makes the algorithm computationally more complex or one has to find a threshold that recovers the true model on the path, respectively. 
In contrast, the first step of the proposed procedure is related only to the Lasso. 
Indeed, we need only that the model with correctly separated predictors is on the path. 
 This is  guaranteed  for the Lasso under mild assumptions. Therefore, it makes our procedure simpler and computationally more efficient.

The paper is organized as follows: in the next section we describe subgaussian GLM  and algorithms that we work with. Moreover, we establish bounds for the selection error of the proposed procedures. These results are extended to models with general convex contrasts in Section \ref{convex}. 
In Section \ref{experiments} we investigate properties of estimators on simulated and real data sets. 
The paper is concluded in Section \ref{conclusions}. All proofs and auxiliary results are relegated to the appendix. 

\section{Subgaussian Generalized Linear Models}%LMs and Fitting Algorithms}
\label{sec:models}

In this section we start with definitions of considered models and estimation criteria. 
Next,  we present model selection algorithms and  state exponential upper  bounds on their selection errors. 

\subsection{Models}
%We first consider a general model model  and discuss imposed assumptions on its parameters. 
The way we model data will encompass normal linear and logistic models as premier examples. Our  assumptions are stated in their most general 
form which allows proving exponential bounds for probability of the selection error without obscuring their essentiality.
In the paper we consider independent data $(y_1,x_1),(y_2,x_2),\ldots,(y_n,x_n)$, 
where $y_i \in \mathbb{R},~ x_i \in\mathbb{R}^p$ for $i=1,2,\dots,n$.
We assume that for some {\it true} $\mathring{\beta} \in \mathbb{R}^p$ and a known differentiable {\it cumulant} function 
$\gamma:\mathbb{R}\rightarrow\mathbb{R}$ 
%a  model parametrized by $\beta \in \mathbb{B}(c_0) =\{\beta \in \mathbb{R}^p : \max_i|x_i^T\beta| \leq c_0 \leq \infty \}$ 
%such that for some $\mathring{\beta} \in \mathbb{B}(c_0)$ 
\begin{equation}
\label{inverseLink}
 \mathbb{E}y_i=\dot\gamma(x_i^T\mathring{\beta})~\text{for}~i=1,2,\dots,n,
\end{equation}
where $\dot \gamma$ denotes the derivative of $\gamma.$
Note that (\ref{inverseLink}) is satisfied in particular by the Generalized Linear Models (GLM) 
with a canonical link function.
and a nonlinear regression with an additive error.
Let $\eta_i=x_i^T\mathring{\beta}$.
 For $\eta=(\eta_1,\ldots,\eta_n)^T$ 
we define $\vec{\gamma}(\eta)=(\gamma(\eta_1),\ldots,\gamma(\eta_n))^T$ and similarly 
$\vec{\dot \gamma}(\eta)=(\dot \gamma(\eta_1),\ldots,\dot \gamma(\eta_n))^T$.

Let $X=[x_{.1},\ldots,x_{.p}]=[x_1,\ldots,x_n]^T$ be a $n\times p$ matrix of experiment and
$J\subseteq \{1,2,\ldots,p\}=F$  be an arbitrary subset of the {\it full model} $F$, ${\bar J}= F\setminus{ J}$.
As $J$ may be viewed as a sequence of zeros and ones on $F$, $|J|=|J|_1$ denotes cardinality of $J$. 
Let $\beta_J$ be a subvector of $\beta$ with elements having indices in $J$, $X_J$ be a submatrix of $X$ with columns having indices in $J$
and $r(X_J)$ denotes the rank of $X_J$. 
Moreover, let $H_J=X_J (X_J^TX^J)^{-1} X_J^T$ be an orthogonal projection matrix onto the subspace spanned by columns of $X_J$.
Linear model pertaining to predictors being columns of $X_J$ will be frequently identified as $J$. 
In particular, let $T$ denotes a {\it true model} that is $T={\rm supp}(\mathring{\beta})=\{j \in F: \mathring{\beta}_j \neq0\}$ and $t=|T|$. Next, for $\beta \in \mathbb{R}^p$ and $q\geq 1$  
let $|\beta|_q = (\sum_{j=1}^p |\beta_j|^q)^{1/q}$ be the $\ell_q$ norm. The only  exception is  the $\ell _2$ norm, for which we will use the special notation $||\beta||.$

In the further argumentation important roles are played by 
\begin{equation}
\label{delta_k}
\delta_k = \min_{J\subset T,\,|T\setminus J|=k } ||(I-H_J)X_T\mathring{\beta}_T||^2,
\end{equation}
for $k=1,\ldots,t-1.$ They describe how much the true value $X \mathring{\beta} = X_T\mathring{\beta}_T$
differs from its projections on submodels of true set $T.$ For the normal linear model they are related to the Kullback-Leibler divergence between two normal densities \citep[Section 3]{PokarowskiMielniczuk15}.
Finally, we define the sum of balls
\begin{equation}
\label{ball}
\mathbb{B}\equiv \mathbb{B}(X,\mathring{\beta},\bar t)= \bigcup_{J\supset T, r(X_J)=|J| \leq 
\bar{t}} \{\beta_J: ||X_J(\mathring{\beta}_J-\beta_J)||^2 \leq \delta_{t-1} \}.
\end{equation}
We assume that $t \leq \bar t < n \wedge p,$ which implies that \eqref{ball} consists of {\it sparse} vectors.

We assume also that a {\it total cumulant} function 
\begin{equation}
\label{total_cumulant}
g(\beta)=\sum_{i=1}^n \gamma(x_i^T\beta)
\end{equation}  
is convex and, additionally, {\it strongly convex} at $\mathring{\beta}$ in a sense that  there exists $c \in (0,1]$ such that for 
all  $\beta \in \mathbb{B} $  we have
\begin{equation}
\label{convexCum}
g(\beta) \geq g(\mathring{\beta}) +(\beta-\mathring{\beta})^T \dot g(\mathring{\beta}) + \frac{c}{2}(\mathring{\beta}-\beta)^T X^TX(\mathring{\beta}-\beta).
\end{equation}
We note that this crucial property of the total cumulant is   slightly weaker than an usual definition of strong convexity which
would have a second derivative of $g$ at $\mathring{\beta}$ in place of $X^TX$.
 Let us remark that 
$\dot g(\beta) = X^T\vec{\dot \gamma}(X\beta)$. 

Moreover, we assume that centred responses $\varepsilon_i = y_i - \mathbb{E}y_i$ have a {\it subgaussian distribution}
with the same  constant $\sigma$, that is for $i=1,2,\dots,n$ and $u \in \mathbb{R}$ we have
\begin{equation}
\label{subgauss}
\mathbb{E}\exp(u\varepsilon_i) \leq \exp(\sigma^2u^2/2).
\end{equation}

{\bf Examples.} Two most important representatives of GLM are the normal linear model and logistic regression. In the normal linear model 
$$
y_i=x_i^T\mathring{\beta} + \varepsilon _i, \quad i=1,2,\dots,n,
$$
where the noise variables $\varepsilon_i$ are independent and normally distributed $N(0,\mathring{\sigma}^2).$ Therefore,  assumptions \eqref{inverseLink}, \eqref{convexCum} and \eqref{subgauss} are satisfied with $\gamma(\eta_i)=\eta_i^2/2, $ $c=1$ and 
 any $\sigma\geq \mathring{\sigma},$ respectively. In logistic regression the response variable 
is dichotomous $y_i
\in \{0,1\}$ and we assume that 
$$
P(y_i=1) = \frac{\exp(x_i^T\mathring{\beta})}{\exp(x_i^T\mathring{\beta})+1}, \quad i=1,\ldots,n.
$$
In this model assumptions \eqref{inverseLink} and \eqref{convexCum}  are satisfied with $\gamma(\eta_i)=\log(1+\exp(\eta_i))$  and \[
c=\min_i \min_{\beta \in \mathbb{B}} \exp(x_i^T\beta)/(1+\exp(x_i^T\beta))^2,
\] 
respectively. Finally, as $(\varepsilon_i)$ are bounded random variables,  then (\ref{subgauss}) is satisfied with any $\sigma\geq 1/2$.
%For fixed $X$ and $\mathring{\beta}$ let $c_0$ be such that $c_0\geq ||X\mathring{\beta}||$ and define $\mathbb{B}(c_0)= \{ \beta: ||X\beta-X\mathring{\beta}|| \leq c_0\}$.
%In case of a linear model we can set $c_0=\infty$ and $c=1$ in (\ref{convexCum}). It is easy to note that for a logistic model the strong convexity assumption 
%(\ref{convexCum}) is fulfilled with $c=\exp(2c_0)/(1+\exp(2c_0))^2$.

\begin{algorithm}[tb]%%[H] na twardo  %[tb] znaczy na miekko
   \caption{TL and SS}
   \label{alg:SS}
\begin{algorithmic}
   \STATE {\bfseries Input:} $y$, $X$ and $\lambda,\tau$
   \STATE {\bfseries Screening} (Lasso) {$~~\widehat{\beta} = \text{argmin}_{\beta} \left\{ \ell(\beta) + \lambda |\beta|_1 \right\}$};
%   \STATE {$\widehat{\beta} = \text{argmin}_{\beta} \left\{ 2\ell(\beta) + 2\lambda |\beta|_1 \right\}$};
 
\medskip
 \STATE {\bfseries Thresholded Lasso} {$~~\widehat{T}_{TL} = \{j: |\widehat{\beta}_j|>\tau\}$};
%   \STATE {$ \widehat{T}_{TL} = \{j: |\widehat{\beta}_j|>\tau\}$}

\medskip
\STATE {\bfseries Selection} (GIC)
\STATE{order nonzero $|\widehat{\beta}_{j_1}| \geq  \ldots \geq |\widehat{\beta}_{j_s}|$, where $s = |\text{supp} \widehat{\beta}|$;}
\STATE{set ${\cal J} = \left\{ \{j_1\},   \{j_1, j_2\}, \ldots, \text{supp} \widehat{\beta}\right\}$;}
\STATE{$\widehat{T}_{SS} = \text{argmin}_{J \in {\cal J}} \left\{ \ell_J + \lambda^2/2 |J| \right\}$}.

\medskip
\STATE {\bfseries Output:} $\widehat{T}_{TL}, \widehat{T}_{SS} $

\end{algorithmic}
\end{algorithm}

\subsection{Fitting Algorithms}
For estimation of  $\mathring{\beta}$ we consider a {\it loss function}
\begin{equation}
\label{loss}
\ell(\beta)=\sum_{i=1}^n [\gamma(x_i^T\beta)-y_ix_i^T\beta] = g(\beta)-\beta^TX^Ty,
\end{equation}
where $y=(y_1,\ldots,y_n)^T$. It is easy to see that $\dot \ell(\beta) = X^T(\vec{\dot \gamma}(X\beta)-y)$,
and consequently $\dot \ell(\mathring{\beta}) = -X^T\varepsilon$ for $\varepsilon=(\varepsilon_1,\ldots,\varepsilon_n)^T$.
Moreover, observe that $\mathring{\beta} = \text{argmin}_{\beta} \mathbb{E}\ell(\beta)$ and (\ref{convexCum}) is equivalent to 
strong convexity of $\ell$ at $\mathring{\beta}$ for all $\beta\in \mathbb{B}$ 
\begin{equation}
\label{convexLoss}
\ell(\beta) \geq \ell(\mathring{\beta}) +(\mathring{\beta}-\beta)^T X^T\varepsilon + 
\frac{c}{2}(\mathring{\beta}-\beta)^T X^TX(\mathring{\beta}-\beta).
\end{equation}
Let $\widehat{\beta}^{ML}_J =\text{argmin}_{\beta_J} \ell(\beta_J)$ denotes a {\it minimum loss estimator} 
based on $y$ and $\{x_{.j}, j \in J \}$. Denote $\ell_J=\ell(\widehat{\beta}^{ML}_J)$. Note that for GLM estimator $\widehat{\beta}^{ML}_J$
coincides with a maximum likelihood estimator for a model pertaining to $J$. 

In Algorithm~\ref{alg:SS} we present two selection procedures: the first one is the well-known Thresholded Lasso (TL)  method which consists of retaining only 
these variables for which absolute values of their Lasso estimators exceed a certain threshold~$\tau$. The second one is novel and named the Screening--Selection (SS) 
procedure. It  finds the minimal value of the Generalized Information Criterion (GIC) for the nested family which is constructed using ordering 
of the nonzero Lasso coefficients.

In the classical (low-dimensional) case  model selection is often based on the Bayesian information criterion that is similar to the second step of Algorithm~\ref{alg:SS} with $\lambda^2=\log n.$ Model selection properties of this procedure are proven in \cite{shao97}. In the current paper we validate that information criteria are also useful in the high-dimensional case. As it will be shown the important difference is that in the high-dimensional scenario parameter $\lambda^2$ should be proportional to $\log p$ instead of $\log n$ (at least for linear models
 with the subgaussian noise).  

%%%%%%%%%%%%%%%%%%%

%\section{Selection Error Bounds for TL and SS}
%\label{selection_error}

\subsection{A Selection Error Bound for TL}%%%%%%%%%%%%%%%%%%%%%%%%%%%%%%%%%%
In order to make the  exposition simpler we assume that columns of $X$ are normalized in such a way that 
$||x_{.j}||=1$ for $j=1,\ldots,p$. Moreover, let $\mathring{\beta}_{min}=\min_{j\in T}|\mathring{\beta}_j|$.

First we generalize a characteristic of linear models which quantifies the degree of separation 
between the true model $T$ and other models, which was introduced in Ye and Zhang \cite{YeZhang10}. For $a\in(0,1)$  consider  a signed pseudo-cone
 \begin{multline}
 \label{cone}
 {\cal C}_a=\bigg\{\nu \in \mathbb{R}^p: |\nu_{\bar T}|_1\leq \frac{1+a}{1-a}|\nu_{ T}|_1,\, \nu_jx_{.j}^T\big[\vec{\dot\gamma}(X(\mathring{\beta}+\nu))-\vec{\dot\gamma}(X\mathring{\beta})\big] \leq 0,\, j\in\bar T\bigg\}.
 \end{multline}
For $q\geq 1$ and $ a \in (0,1)$ let a Sign-Restricted Pseudo-Cone Invertibility Factor (SCIF)  be defined as
 \begin{equation}
 \label{SCIF}
 \zeta_{a,q}=\inf_{0 \neq \nu\in {\cal C}_a}\frac {\big|X^T\big[\vec{\dot\gamma}(X(\mathring{\beta}+\nu))
                                                     -\vec{\dot\gamma}(X\mathring{\beta})\big]\big|_{\infty} } {|\nu|_q}\:.
%       =  \inf_{\nu\in {\cal C}_a}\frac {|\dot g(\mathring{\beta}+\nu)-\dot g(\mathring{\beta})|_{\infty}} {|\nu|_q}.                                          
 \end{equation}
Notice that in the linear model the numerator of \eqref{SCIF} is simply  $|X^TXv|_\infty.$ 
In the case $n>p$ one usually uses the minimal eigenvalue of the matrix $X ^TX$ to express the strength of correlations between predictors. Obviously, in the high-dimensional scenario this value is zero. 
Therefore, SCIF can be viewed as an useful analog of the minimal eigenvalue for the case $p>n.$

We  let $\zeta_a=\zeta_{a,\infty}$. In comparison to more popular restricted eigenvalues \citep{BickelEtAl09} or compatibility constants \citep{GeerBuhlmann09}, variants of SCIF enable sharper 
$\ell_q$ estimation error bounds of the Lasso for $q>2$  \citep{YeZhang10, HuangZhang12, ZhangZhang12}. 

The following lemma is a main tool in proving model selection consistency of the TL algorithm. For the linear model it was stated in \citep[Theorem 3]{YeZhang10}. We generalize it to GLM. 

\begin{lemma}
\label{LassoBound}
If $\ell$ is convex and $a \in (0,1)$, then on $\{|X^T\varepsilon|_\infty\leq a\lambda\}$ we have 
$$ |\hat{\beta}-\mathring{\beta}|_q \leq (1+a)\lambda\zeta_{a,q}^{-1} $$.
\end{lemma}

Next, we bound the selection error of the TL algorithm in GLM.
\begin{theorem}
\label{th1}
If $\ell$ is convex, $(\varepsilon_i)_i$ are subgaussian with $\sigma$ and for numbers $a_1,a_2 \in (0,1)$  we have 
\begin{equation*}
2a_1^{-2}a_2^{-1}\sigma^2\log p \leq \lambda^2 
 \leq (1+a_1)^{-2}\zeta_{a_1}^2\tau^2 < (1+a_1)^{-2}\zeta_{a_1}^2\mathring{\beta}_{min}^2/4,
\end{equation*}
then
 \begin{equation}
 \label{corEq1}
 \mathbb{P}(\hat{T}_{TL}\neq T)\leq 2\exp\bigg(-\frac{(1-a_2)a_1^2\lambda^2}{2\sigma^2}\bigg).
 \end{equation}

\end{theorem}
Constant $a_2$ is used to remove multiplicative factor $p$ from the exponential  bound at  the expense of slightly diminishing  the exponent 
in (\ref{corEq1}). Note that assumptions of Theorem \ref{th1} stipulate that the truncation level $\tau$ is contained in the interval 
$[(1+a_1)\lambda\zeta_{a_1}^{-1},\mathring{\beta}_{min}/2)$,  whose both endpoints are
unknown, because $\zeta_{a_1}$ and $\mathring{\beta}_{min}$ are unknown. Therefore, in practice $\tau$ cannot be chosen that makes the TL algorithm non-constructive. 
Analogous theorems for FCP for linear models and logistic regression can be found in Fan et al. \cite[Corollary~3 and Corollary~5 ]{FanEtAl14}. However, they require an additional  assumption on  the minimal eigenvalue of $X_T^TX_T$ and the proof is more difficult. Moreover, the choice of the tuning parameters in these methods also requires unknown $\zeta_{a}$ or $\mathring{\beta}_{min}$.

\subsection{A Selection Error Bound for SS}%%%%%%%%%%%%%%%%%%%%%%%%%%%%%%%%%%%%%%%%%%%%%%
\label{selection_error_GLM}

A scaled Kullback-Leibler (K-L) distance between $T$ and its submodels is given in  \citet{ShenEtAl12, ShenEtAl13} as
 \begin{equation}
 \label{delta}
\delta=\min_{J\subset T}\frac{||(I-H_J)X_T\mathring{\beta}_T||^2}{|T \setminus J|}
 \end{equation}
or just $\min\limits_{k=1,\ldots, t-1} \delta_k/k,$ if we use notation \eqref{delta_k}.
Different variants of the K--L distance have been often used in the consistency analysis of selection algorithms \citep[Section 3.1]{PokarowskiMielniczuk15}, but $\delta$ defined in \eqref{delta} seems to lead to optimal results  \citep[Theorem~1]{ShenEtAl13}.
Now we introduce technical constants $a_1,\ldots,a_4$  that allow to avoid {\it ad-hoc} coefficients in the main results and simplify asymptotic considerations.
%statement of asymptotic consistency property. 
%W ponizszym tw podajemy oszacowanie bledu selekcji algorytmu SS. Zmienne dodatkowe (oslabiacze) $a_1,\ldots,a_4$ pozwalaja uniknac stalych 
%{\it ad-hoc} oraz upraszczaja rozwazania asymptotyczne (wnioski...) 
For given $1/2<a_1<1$ define $a_2=1-(1-\log (1-a_1))(1-a_1)$, 
$a_3=2-1/a_1$ and $a_4=\sqrt{a_1a_2}$. Note that $a_2,a_3$ and $a_4$ are functions of $a_1$ and  obviously if $a_1 \rightarrow 1$, then  
$a_2,a_3,a_4 \rightarrow 1$. Moreover, for two real numbers $a,b$ we denote $a \vee b = \max(a,b)$ and $a \wedge b = \min(a,b).$

%%%%%%%%%%%%
\begin{theorem}
\label{th2}
Assume (\ref{inverseLink}), (\ref{convexCum}), (\ref{subgauss}) and that for $a_1 \in (1/2,1)$ 
we have
\begin{equation}
\label{lambda_glm}
\frac{2\sigma^2\log p}{a_3a_2a_1c} \vee \frac{\sigma^2t}{(1-a_1)^2c} \leq \lambda^2 < 
\frac{c \dt  }{16(\bar{t}-t)} \wedge  \frac{c \delta }{(1+\sqrt{2(1-a_1)})^2} \wedge \frac{\zeta_{a_4}^2\mathring{\beta}_{min}^2}{4(1+a_4)^2}.
\end{equation}
 Then
 \begin{equation}
 \label{main2}
 \mathbb{P}(\hat T_{SS} \neq T)\leq 4.5\exp\bigg(-\frac{a_2(1-a_1)c\lambda^2}{2\sigma^2}\bigg).
 \end{equation}

\end{theorem}

Selection consistency, that is asymptotic correctness of $\hat{T}_{SS},$ now easily follows.
\begin{corollary}
\label{cor2}
Assume that $t=o(\log p)$ for $n \rightarrow \infty.$ Moreover, set $a_1=1-\sqrt{\frac{t}{2 \log p}}$ and
$\lambda^2=\frac{2\sigma^2 \log p}{a_3a_2a_1c}$. 
Then $\lambda^2=c^{-1}(2\sigma^2 \log p)(1+o(1))$. 
If additionally $\mathring{\beta}$ is asymptotically identifiable, i.e.
\[\varlimsup_n \frac{2\sigma^2 \log p/c}{\frac{c\dt}{16(\bar{t}-t)} \wedge (c\delta) \wedge (\zeta_{a_4}^2\mathring{\beta}_{min}^2/16)} < 1, 
~~~\text{then}~~~ \mathbb{P}(\hat T_{SS}\neq T)=o(1).
\]
\end{corollary}

%{\bf Remark.} 
%Note that assumptions of Corollary \ref{cor2} imply that $p$ tends to infinity. For constant $p$ it follows from Theorem \ref{th2} that
%$T_{SS}$ is consistent provided ...
%\begin{remark}
%\rm{ In the next subsection Theorem \ref{th2}  for linear models  with the subgaussian noise (Theorem \ref{th_LM}) is obtained directly,  analogously as in \cite{PokarowskiMielniczuk15}. The resulting 
%lower bound on $\lambda^2$ is proportional to $2\sigma^2\log p$  without the additional condition 
%$\lambda^2\geq \sigma^2t/(1-a_1)^2$ assumed in Theorem \ref{th2} (recall that $c=1$). 
%Thus for  linear models the proposed algorithm is constructive in contrast to TL and FCP.
%}
%\end{remark}

\begin{remark}
\label{safest_choice}
\rm{
Theorem~2 determines conditions on GLM and the SS algorithm for which the bound 
(\ref{main2}) on the selection error of SS holds.
Corollary~1 states the easy interpretable result: if the true model is 
asymptotically identifiable, then SS with minimal
admissible $\lambda$ is asymptotically consistent. Although the identifiability 
condition is not effectively verifiable,
$\lambda$ can be explicitly given for linear models as
\begin{equation}
  \label{star1}
\lambda=\sqrt{2\sigma^2\log p}(1+ o(1))
\end{equation}
and for logistic models as
\begin{equation}
  \label{star2}
  \lambda=\sqrt{(\log p)/(2c)}(1+ o(1)),
\end{equation}
since $\sigma \geq 1/2$.
Let us consider subgaussian linear models and assume  that $\sigma^2$ is known, which   is a common condition in the literature investigating theoretical properties of variable selection procedures \cite{YeZhang10, BuhlmannGeer11,FanEtAl14}. 
Then  the parameter $\lambda$ of SS is given constructively provided that
$t=o(\log p)$. In  contrast, TL or FCP are not constructive, because they require an additional parameter $\tau$, that
depends on unknown identifiability constants as SCIF.

In the literature concerning the Lasso and 
its modifications the smallest possible $\lambda$ is taken
as the default value, because it makes the algorithm  asymptotically consistent 
for the largest class of models (the same approach is adopted for
prediction and estimation). Such $\lambda$ will be called the {\it safest 
choice}.
It is interesting that for linear models GIC with $\lambda$ given by (\ref{star1}) was originally derived from the minimax  perspective 
by Foster and George \cite{FosterGeorge1994}. They called such selection the risk inflation criterion (RIC), because it asymptotically minimises 
the maximum predictive risk with respect to the oracle for the orthogonal matrix of experiment $X$.
}
\end{remark}

\begin{remark}
\rm{
A generic combination of the penalized log-likelihood (as TL or FCP) with 
GIC is considered in \cite{FanTang2013}.
In the first step the method computes a path of models indexed by 
$\lambda$ and next GIC is used to choose the final model.
They assume that the true model has to be on this path and use GIC with the
penalty asymptotically larger than $\log p$.
Thus, their results are weaker and need more restrictive assumptions, which are 
given in \cite[Section 6]{FanTang2013}.
For instance, if the cumulant function $\gamma,$ defined in 
\eqref{inverseLink}, has uniformly bounded second derivative,
then we do not require its third derivative in contrast to 
\cite{FanTang2013}. Moreover, using Corollary \ref{cor2}
and assuming that the number $c$ is constant as in \cite{FanTang2013}, 
the last step of our algorithm uses GIC with
the safest choice $\lambda$ instead of $K \log p$ as in 
\cite{FanTang2013} for some $K \rightarrow \infty$ with  $n \rightarrow 
\infty$.
It is worth  to note that their results are obtained using the empirical 
processes theory, while the proof of Theorem \ref{th2} is based on
elementary exponential inequalities for subgaussian variables given in 
subsection~\ref{subsec:subgaussian}.
}
\end{remark}

\subsection{A Selection Error Bound for SS in Subgaussian Linear Models}

In this part of the paper we show that SS is constructive for the linear 
model  with the subgaussian noise.
The main difference between Theorem \ref{th2} and the following Theorem 
\ref{th_LM} is that the lower bound on $\lambda^2$ in Theorem \ref{th_LM}
does not depend on the dimension of the true model $t$ and the parameter $c$ (because  $c=1$ 
for subgaussian linear models).

\begin{theorem}
\label{th_LM}
Consider the linear model with the subgaussian noise. Assume that there 
exists $0<a \leq \sqrt{1-(1+\log(2\log p))/(2\log p)}$   such that
  $$\frac{2\sigma^2\log p}{a^4 }\leq \lambda^2 < \frac{\zeta_{a}^2\bmin 
^{2} \wedge \delta}{4(1+a)^{2}}.$$ Then
  \begin{equation}
  \label{mainLM}
  P(\hat T_{SS}\neq T)\leq 
5\exp\bigg(-\frac{a^2(1-a^2)\lambda^2}{2\sigma^2}\bigg).
  \end{equation}
  \end{theorem}

The safest choice of $\lambda$ in Theorem \ref{th_LM} does not depend on 
unknown expressions, which justifies the claim that our algorithm is 
constructive
in the linear model with the subgaussian noise. Next, we compare the above 
result to \citet{WangEtAl13} and \citet{FanEtAl14}.

\begin{remark}
\rm{
The algorithm in \citet{WangEtAl13} has three steps: the Lasso, non-convex 
penalized linear regression and GIC with the parameter $C \log p,$
where $C \rightarrow \infty$ with  $n \rightarrow \infty$. 
Obviously, the SS algorithm is computationally faster, because it does not need the most 
time-consuming second step. Moreover, the first two steps of
their algorithm form a variant of FCP, so their parameters are not given constructively as we explain in the discussion after Theorem \ref{th1}. Finally, their assumptions are stronger than ours. Indeed, conditions
in \citet[Theorem 3.2, Theorem 
3.5]{WangEtAl13}   lead to $n/(\delta t) =O(1)$ and  $t \log p/n= o(1),$ where $\delta$ is defined in \eqref{delta}. From these two facts and \citet[expression (3.4)]{WangEtAl13} we obtain 
\begin{equation}
\label{wang}
\frac{\log p}{\delta} =o(1) \quad {\rm and } \quad \frac{t \sqrt{\log p}}{\bmin}=o(1).
\end{equation}
If we fix  $\sigma ^2$ and $\zeta_{a}^2$ as in \citet[expression (3.4)]{WangEtAl13}, then  in Theorem \ref{th_LM} we require only that
\begin{equation*}
%\label{us}
\frac{\log p}{\delta} =O(1) \quad {\rm and } \quad \frac{\sqrt{\log p}}{\bmin}=O(1).
\end{equation*}
In \citet[Theorem 3.6]{WangEtAl13}  the second condition in \eqref{wang} is weakened at the price of stronger conditions on the the design matrix. However, their  assumptions are still 
more restrictive assumptions
 than ours.
}
\end{remark}

\begin{remark}
\rm{
\citet[Subsection 3.1]{FanEtAl14} consider an algorithm based on non-convex penalization in linear models. 
Their procedure is model selection consistent, but is also  not 
constructive.
Indeed, in \citet[Corollary 3]{FanEtAl14} the parameter $\lambda$ has to 
be in the interval, whose both endpoints depends on unknown
$||\beta_{\mathcal{A}}^*||_{min},\sqrt{s}, \kappa_{linear}$ that are 
analogs of $\bmin , t, \xi_{a_1}$ in the current paper.
In \citet[Remark 5]{FanEtAl14} it is shown that $\sqrt{s}$  can be 
avoided, but their algorithm is still not constructive in contrast to the 
SS algorithm.
}
\end{remark}

\subsection{Exponential Bounds for Subgaussian Vectors}
\label{subsec:subgaussian}

This part of the paper is devoted to  exponential inequalities for subgaussian random vectors. They are interesting by themselves and can be  used in different problems than we consider. In the current paper they are main probabilistic tools that are needed to prove Theorems \ref{th1}-\ref{th_LM}.
 Specifically, in lemma~\ref{lemSubG}~(iii) we generalize the Wallace inequality for 
$\chi^2$ distribution \cite{Wallace59} to the subgaussian case using the inequality for the moment generating function in lemma~\ref{lemSubG}~(ii). 
The last inequality is proved by the decoupling technique as in the proof of Theorem~2.1 in \cite{HsuEtAl12}.

\begin{lemma}
 \label{lemSubG}
Let $\varepsilon \in \mathbb{R}^n$ be a vector of zero-mean independent errors having subgaussian distribution with 
a constant $\sigma$,  $\nu \in \mathbb{R}^n$, $0<a<1$ and $H$ be an orthogonal projection such that $tr(H)=m$. Then 
\begin{itemize}

\item [(i)] for $\tau>0$
 \begin{equation}
\label{lin_ineq} \mathbb{P}(\varepsilon^T\nu/||\nu|| \geq \tau) \leq \exp\left(-\frac{\tau^2}{2\sigma^2}\right),
\end{equation} 

\item [(ii)] $$\mathbb{E}\exp\left(\frac{a}{2\sigma^2}\varepsilon^T H\varepsilon\right) \leq \exp\left(-\frac{m}{2}\log(1-a)\right),$$

%(iii)~~$ \mathbb{P}(\varepsilon^T H\varepsilon \geq \tau) 
%\leq \exp\big(-\frac{m}{2}\big(\frac{\tau}{m\sigma^2}-1-\log(\frac{\tau}{m\sigma^2})\big)\big)~\text{for}~\tau>m\sigma^2 $.
\item [(iii)] for $\tau>1$ 
\begin{equation}
\label{quadr_ineq}
 \mathbb{P}(\varepsilon^T H\varepsilon \geq m\sigma^2\tau) 
\leq \exp\left(-\frac{m}{2}\big(\tau-1-\log\tau\big)\right).
\end{equation}
\end{itemize}
\end{lemma}%Proof of lemma \ref{lem1}.

\section{Extension to General Convex Contrasts}
\label{convex}

In this part of the paper we investigate properties of  the SS algorithm beyond GLM as well. The main assumption, that will be required,  is convexity of the ,,contrast'' function. 
We show that the SS algorithm is  very flexible procedure that can  be  applied succesfully to the various spectrum of practical problems. 

 First, 
for $\beta \in \Rp$ and a contrast function $\phi:\R \times \R \rightarrow \R$ we define a loss
function  
$$
\ell(\beta) =  \sum_{i=1}^n \phi(\beta ^T x_i, y_i).
$$
Considering the normal linear model one usually uses the quadratic contrast $$\phi(\beta ^T x_i, y_i)=(y_i-\beta ^T x_i)^2$$ as we have done in Section \ref{sec:models}. However, it is well known that the quadratic contrast is very sensitive to the distribution of errors $\varepsilon _i$ and does not work well, if this distribution is heavy-tailed and outliers appear. To overcome this difficulty we can use the absolute contrast $$\phi(\beta ^T x_i, y_i)=|y_i-\beta ^T x_i|.$$ Next, working with dichotomous $y_i$ we can apply  logistic regression that belongs to GLM and has been considered in Section \ref{sec:models}. In this case we have $$\phi(\beta ^Tx_i, y_i) = -y_i\beta ^T x_i + 
\log [1+\exp(\beta ^T x_i)].$$ But there are also very popular and efficient algorithms called  support vector machines (SVM) that use, for instance, the following quadratic hinge  contrast $$\phi(\beta ^Tx_i, y_i)=[\max(0,1-y_i\beta ^T x_i)]^2.$$ 
Our main assumption is that the contrast function $\phi$ is convex with respect to $\beta.$ All examples given  above  satisfy this property. Notice that they need not  be differentiable nor decompose, as in \eqref{loss} for GLM, into the sum of the nonrandom cumulant $\gamma$  and the random linear term $y_i \beta ^T x_i$. The SS algorithm for convex contrasts is the same as in Algorithm 1.

We add few definitions and notations to those in the previous parts of the paper. 
 We start with defining two balls: the first one is the $l_1$-ball
$B_1(r) =\{\beta: |\beta - \bo|_1 \leq r \}$ with radius $r>0.$ The second one is the $l_2$-ball
$
B_{2,J}(r) =\{\beta_J: ||X_J (\bo-\beta_J)||^2 \leq r^2\} 
$
with radius $r>0,$ where $J$ is a (sparse) subset of  $\{1, \ldots,p\}$ such that $T \subset J, r(X_J)=|J|\leq \bar{t}.$ Recall that $\bo$ is, as previously, a minimizer of  $\Ex \ell(\beta).$
Besides, let $B_J=B_{2,J} \left(\sqrt{\dt }\right),$ where $\dt $ is defined in \eqref{delta_k}. 
In further argumentation  key roles are played by: 
$$
Z(r) = \sup_{\beta \in B_1(r)} \left|\ell(\beta)- \Ex \ell(\beta) -[\ell(\bo)-\Ex \ell(\bo)  ]\right| $$
and
$$
U_J(r) = \sup\limits_{\beta \in B_{2,J}(r)} \left|\ell(\beta)- \Ex \ell(\beta) -[\ell(\bo)-\Ex \ell(\bo)]  \right|,
$$
which are empirical processes over $l_1$ and $l_2$-balls, correspondingly. We need also the compatibility factor that is borrowed from \cite{BuhlmannGeer11} and is an analog of SCIF defined in \eqref{SCIF}. Namely,
for arbitrary $a \in  (0,1)$ a compatibility factor is 
\begin{equation}
\label{comp_fact}
\kappa_a = \inf_{0 \neq \beta \in \mathcal{C}_a}  \frac{ \beta ^T X^TX \beta }{|\beta_{T}|_1^2} \: ,
\end{equation}
where $\mathcal{C}_a$ is a simplified version of  \eqref{cone}, namely 
$$
 {\cal C}_a=\bigg\{\nu \in \mathbb{R}^p: |\nu_{\bar T}|_1\leq \frac{1+a}{1-a}|\nu_{ T}|_1 \bigg\}.
 $$

Convexity of the contrast function is the main assumption in this section. However, similarly to the previous section we need also the following {\it strong convexity} of $\Ex \ell(\beta)$
 at $\b*$:
there exists $c_1 \in (0,1]$  ($c_2 \in (0,1]$, respectively) such that for each $\beta_1 \in B_1(\bmin)$ ($\beta_2 \in \mathbb{B}$ defined in \eqref{ball}, respectively) we have for $i=1,2$
\begin{equation}
\label{strong_c}
\Ex \ell(\beta_i) -\Ex \ell(\b*) \geq \frac{c_i}{2} (\beta_i - \b*) ^T \, X^TX \, (\beta_i - \b*) .
\end{equation}
Notice that in \eqref{strong_c} we  require  the expected loss $ \Ex \ell(\beta),$ not the loss $\ell(\beta),$ to be strongly convex. 
Therefore, the condition \eqref{strong_c} can be  satisfied  easily even if the contrast function $\phi$ is not differentiable, for instance for  absolute or quadratic hinge contrasts (see Remark \ref{32}). 
For GLM in section \ref{sec:models} the condition \eqref{convexLoss} is equivalent to \eqref{strong_c} for $i=2,$ that will be explained in Remark \ref{comp_similar}.

 To prove exponential bounds for GLM in subsection \ref{selection_error_GLM}  we use subgaussianity that allows us to obtain probabilistic inequalities in Lemma \ref{lemSubG}.  In this section we need the analog of \eqref{lin_ineq} of the form: there exists $L>0$ and constants $K_1,K_2>0$ such that for each $0<r\leq \bmin$ and $z \geq 1$ 
we have 
\begin{equation}
\label{ZM}
P\left( \frac{Z(r)}{r} > K_1  L  z \sqrt{ \log (2p)} 
\right)\leq \exp\left(- K_2 \log (2p) \:z^2 \right).
\end{equation}
Besides, the inequality \eqref{quadr_ineq} is replaced by the following: there exists $L>0$ and constants $K_3,K_4>0$ such that for each $0<r\leq \sqrt{\dt} ,\;z\geq 1$ and $J$ such that $T \subset J, r(X_J)=|J|\leq \bar{t}$
we have 
\begin{equation}
\label{UJ}
P\left( \frac{U_J(r)}{r} >K_3L  z \sqrt{|J|}
\right)\leq \exp(-K_4 |J|z^2).
\end{equation}
The detailed comparison between assumptions and results for models in this section and those for GLM in Theorem \ref{th2} is given in Remarks \ref{comp_similar} and \ref{comp_diff} after the main result of this section, which is  now stated. 

\begin{theorem}
\label{main_emp}
Fix $a_1,a_2 \in (0,1)$ and let $K_i$ be universal constants. Assume  that \eqref{strong_c}, \eqref{ZM}, \eqref{UJ} and
\begin{eqnarray}
\label{lambda_assum}
&\,& K_1 \max\left(\frac{\log p}{a_1^2},\frac{ \log p}{c_2}, \frac{t}{a_2c_2}\right)  L^2  \leq \lambda^2 \leq \\
\label{lambda_assum2}
&\leq&  K_2\min \left[
\frac{c_2 \delta}{(1+\sqrt{2a_2})^2} , \frac{c_2 \dt }{\bar{t} -t}
, (1-a_1)^2 c_1^2 \kappa^2_{a_1} \bmin ^2 \right].
\end{eqnarray}
Then 
\begin{equation}
\label{claim}
P\left( \hat{T}_{SS} \neq T 
\right) \leq K_3 \exp\left[- \frac{K_4 \lambda^2}{L^2}\min\left(a_1^2 ,
  a_2  c_2 \right)\right].
\end{equation}
\end{theorem}

Theorem \ref{main_emp} bounds exponentially the selection error of the SS algorithm.
It extends Theorem \ref{th2} to the wide class of  convex contrast functions. In particular,
these contrasts can be nondifferentiable as in quantile regression or SVM. In Remarks  \ref{ZMUJ} and \ref{32} we discuss assumptions  \eqref{ZM}, \eqref{UJ} and \eqref{strong_c} of Theorem \ref{main_emp}, respectively. The detail comparison to Theorem \ref{th2} is given in Remarks \ref{comp_similar} and \ref{comp_diff}. There we argue that 
Theorem \ref{main_emp} applied to GLM is only slightly worse than Theorem
 \ref{th2}, which is devoted to GLM. 

\begin{remark} \rm{
\label{ZMUJ}
The important assumptions of Theorem \ref{main_emp} are  conditions \eqref{ZM} and \eqref{UJ}. 
They can be proved using tools from the empirical process theory such that concentration inequalities \citep{Massart2000}, the Symmetrization Lemma \citep[Lemma 2.3.1]{vw:96} and the Contraction Lemma \citep[Theorem 4.12]{ledtal:91}. It is quite remarkable that to get \eqref{ZM} or \eqref{UJ} we need only one new condition. Namely, we need  that the contrast function is Lipschitz in the following sense: there exists $L>0$ such that for all $x_i,y,0<r\leq \bmin$ and $\beta, \tilde{\beta} \in B_1(r)$ 
\begin{equation}
\label{Lip}
|\phi(\beta ^Tx_i,y) - \phi(\tilde{\beta} ^T x_i,y)| \leq L |\beta ^T x_i - \tilde{\beta} ^Tx_i|.
\end{equation} 
Indeed, \eqref{ZM} with $K_1=8 \sqrt{2}, K_2= 4$ follows from the above-mentioned tools and can be established as in  \cite[Lemma 14.20]{BuhlmannGeer11} combined with \cite[Theorem 9]{Massart2000}.
On the other hand, to get \eqref{UJ} with $K_3=8, K_4=2$ we need \eqref{Lip} to be satisfied for all
$x_i,y,0<r \leq \sqrt{\dt}, J:T \subset J, r(X_J)=|J|\leq \bar{t}$ and $\beta, \tilde{\beta} \in B_J.$
This fact can be obtained as in \cite[Lemma 14.19]{BuhlmannGeer11} combined again with \cite[Theorem 9]{Massart2000}.
Notice that logistic and absolute contrast functions satisfy \eqref{Lip} with $L=2$ and $L=1,$ respectively. The property \eqref{Lip} is also satisfied for the quadratic hinge contrast, but in this case $L$  depends on $n$.
}
\end{remark}

\begin{remark} \rm{
\label{32}
The condition \eqref{strong_c} is often called the ''margin condition'' in the literature. 
For quadratic and logistic contrasts we have  considered it in the previous section. To prove it for  SVM with the quadratic hinge contrast one can use methods based on the modulus of convexity \citep[Lemma 7]{bja:06}. For linear models with the absolute  contrast it can be established analogously to \cite[Lemma 3]{KimJeon2016}, if densities of noise variables are lower-bounded in a 
neighbourhood of the origin.}
\end{remark}

\begin{remark} Similarities to Theorem \ref{th2}.  \rm{  
\label{comp_similar}
 We compare  Theorem \ref{th2} to Theorem \ref{main_emp} applied to GLM. 
 We can calculate that for quadratic and logistic contrasts we have 
$$\ell(\beta) - \ell(\bo) = - (\beta-\b*) ^T X^Ty + g(\beta) - g(\b*)$$
and 
$$\Ex \ell(\beta) - \Ex \ell(\bo)  = - (\beta-\b*) ^T X^T \Ex y + g(\beta) - g(\b*),$$
where $g$ is a total cumulant function \eqref{total_cumulant}. Therefore, the condition \eqref{strong_c} for $l_2$-balls is the same as \eqref{convexLoss}.
Besides, we have for $\varepsilon = y- \Ex y$ that 
\begin{equation}
\label{vsGLM}
\ell(\beta) - \Ex \ell(\beta) -[\ell(\bo) - \Ex \ell(\bo)]=(\beta-\b*) ^T X^T\varepsilon.
\end{equation}
In Theorem \ref{th2} we supppose that $\varepsilon_1, \varepsilon_2, \ldots, \varepsilon_n$ are independent and subgaussian. These assumptions are used to establish \eqref{lin_ineq} and \eqref{quadr_ineq}, that are crucial in the proof of Theorem \ref{th2}.
Notice that \eqref{vsGLM} implies that for GLM we have $Z(r)/r \leq  \left|X^T \varepsilon\right|_\infty$ and $U_J(r)/r \leq  \sqrt{ \varepsilon ^T H_J \varepsilon}.$
Therefore, assumptions \eqref{ZM} and \eqref{UJ} are analogs of \eqref{lin_ineq} and \eqref{quadr_ineq}, respectively. 
 Moreover,  the  condition \eqref{lambda_assum} and the result in \eqref{claim}  
differ only in constants from their counteparts in Theorem \ref{th2}. }  
\end{remark}

\begin{remark} Differences from Theorem \ref{th2}.  \rm{
\label{comp_diff}
The SS algorithm consists of two steps. In the last paragraph we have clarified that the theoretical analysis of the second step (selection)  in
Theorem \ref{th2} is not significantly simpler than for general models with convex contrasts. However, we can find differences while investigating the first step (screening based on the lasso). 
Working with GLM we exploit  differentiability of contrasts and the useful decomposition \eqref{loss}. Due to that the right-hand side of \eqref{lambda_glm} is usually better than \eqref{lambda_assum2}, because in Theorem \ref{main_emp} we have to assume \eqref{strong_c} also   
with respect to $l_1$-balls and 
 $c_1^2$ appears in \eqref{lambda_assum2}. 
}
\end{remark}

\section{Experiments}
\label{experiments}

While convenient for theoretical analysis TL, FCP or SS algorithms use the Lasso estimators only for one value of the penalty, 
the practical Lasso implementations return coefficient estimators for a given net of it as in the R~package {\tt glmnet} described 
in \cite{FriedmanEtAl10}. Similarly, using a net of penalty values, the Minimax Concave Penalty (MCP) algorithm, a popular realization of FCP, 
has been implemented for linear  and logistic models in the R~packages {\tt SparseNet} \cite{MazumderEtAl11} and  {\tt ncvreg} \cite{ncvreg}), respectively.
Thus, we propose a net modification of the SS algorithm, which we call SSnet and state in Algorithm \ref{alg:SOSnet}. 
In the first step this procedure calculates the Lasso for a net of $m$ values of {\it an input grid}: $\lambda^L_1, \ldots, \lambda^L_m$.  Then the final model is chosen using GIC in a similar way  to the SS algorithm.

 We remind that for linear models $\lambda$ depends on the noise variance, which is unknown in practice. 
Estimation of $\sigma ^2$ is quite difficult, especially in the high-dimensional
scenario. Computing the Lasso for the whole net of tuning parameters is a simple way to overcome this problem. Obviously, in the second step of our procedure we use GIC, so we require an estimator of $\sigma ^2.$ But in this step we work with  models, whose dimensionality is significantly
reduced by the Lasso. Therefore, the classical estimator of $\sigma^2$, which uses residual
sum of squares, can be applied. 

\begin{algorithm}[h]%%[tb]   
   \caption{SSnet}
   \label{alg:SOSnet}
\begin{algorithmic}
   \STATE {\bfseries Input:} $y$, $X$ and $({\rm input \; grid:}  \lambda_1^L < \ldots < \lambda_m^L, {\rm output \; grid:} \lambda_1 < \ldots < \lambda_r )$
   \STATE {\bfseries Screening} (Lasso)
   \FOR{$k=1$ {\bfseries to} $m$}
   \STATE {$\widehat{\beta}^{(k)} = \text{argmin}_{\beta} \left\{ \ell(\beta) + \lambda_k^L |\beta|_1 \right\}$};
   \STATE {order  nonzero $|\widehat{\beta}_{j_1}^{(k)}| \geq \ldots \geq |\widehat{\beta}_{j_{s_k}}^{(k)}|$, where $s_k = |\text{supp} \widehat{\beta}^{(k)}|$;}
 \STATE {set ${\cal J}_{k} = \left\{ \{j_1\},   \{j_1, j_2\}, \ldots, \{j_1, j_2, \ldots, j_{s_{k}}\} \right\}$}
\ENDFOR;

    \STATE {\bfseries Selection} (GIC)
    \STATE {${\cal J} = \bigcup_{k=1}^m   {\cal J}_{k}$ }; 
    \FOR{$k=1$ {\bfseries to} $r$}
     \STATE {$\widehat{T}_k = 
                     \text{argmin}_{J \in {\cal J}} \left\{ \ell_J + \lambda_k^2/2 |J| \right\}$};

\ENDFOR;
\STATE {\bfseries Output:} $\widehat{T}_1,\ldots,\widehat{T}_r.$

\end{algorithmic}
\end{algorithm}

We performed numerical experiments fitting sparse linear and logistic models to high-dimensional benchmark  simulations and real data sets. We investigated properties in model selection and prediction of the SSnet algorithm and its competitor, which was a net version of MCP. The fair comparison of these two net algorithms is difficult, because their input grids depend on these algorithm themselves  and their runs. Therefore, we decide to introduce {\it an output grid}, which is the same for both algorithms  $\lambda_1 < \ldots < \lambda_r$ in Algorithm~\ref{alg:SOSnet}. We compare the algorithms using graphs, which describe the interplay between a prediction error (PE)  and a model dimension (MD). In these graphs we show the averaged PE and MD over simulation runs (data sets) for distinct values from the output grid.  In particular, we want to find the minimum PE and the optimal penalty. Basing only on input grids, it would be more difficult to interpret  such averaging  as well as  to locate the PE minimum.
Finally, using such plots makes our procedure more {\it user-friendly}. Namely, observing the relation between dimensionality of a model and its prediction accuracy a user can decide, which estimator is best suited for both: describing the phenomena and his/her expectations.

\subsection{Simulated Data}

%We performed two sets of experiments, the first one based on example 1  in \cite{WangEtAl13}, the second based on the second experiment in \cite{WangEtAl14}. 

For linear models we studied the performance of two algorithms: SSnet and MCP  computed using the R package {\tt SparseNet} \cite{MazumderEtAl11} 
for the default 9 values of $\gamma$ and 50 values of $\lambda$. Our algorithm  used the R package {\tt glmnet} \cite{FriedmanEtAl10} 
to compute the Lasso estimators for 50 lambdas on a  log scale.
%The maximal value of  $\lambda$s was $|2*X^Ty/n|_{\infty}$, the minimal was .0001 times smaller and the rest were evenly 
%spaced on a log scale giving the total of 50 values, the SOSnet parameter equaled $o = 5$.  

We generated  samples $(y_i, x_i)$, $i =1, \ldots, n$  from the normal linear model. Two vectors of parameters 
were considered: \mbox{$\mathring{\beta}^{(1)} = (3, 1.5, 0, 0, 2, 0_{p-5}^T)^T,$} as in  \cite{WangEtAl13} as well as
\mbox{$\mathring{\beta}^{(2)}= (0_{p-10}^T, s_1 \cdot 2, s_2 \cdot 2, \ldots, s_{10} \cdot 2)^T$}, where $s_l$ equals 1 or -1 with equal probability, 
$l = 1, \ldots,10$ chosen separately for every run as in experiment~2 in \cite{WangEtAl14}. 
The rows of $X$ were iid $p$-dimensional vectors $x_i \sim N(0_p, \Xi)$. We considered auto-regressive structure of covariance matrix that is 
$\Xi= \big(\rho^{|i-j|}\big)_{i,j = 1}^p$ for $\rho=0.5,0.7,0.9$. The columns of $X$ were centred and  normalized so that $||x_{\cdot j}||^2=n$ 
and $\varepsilon \sim N(0_n,\sigma^2I_n)$. The plan of experiments is presented in Table~\ref{T1} with SNR meaning a {\it Signal to Noise Ratio}.
\begin{table}[tb]
\centering
\caption{Plan of experiments for linear models.}
\label{T1}
\begin{small}
\begin{tabular}{lcccccccccc}
  \hline
 & n & p & $\mathring{\beta}$ & $\rho$  & $\sigma^2$ & SNR \\
\hline
N.1.5 & 100 & 3000 & $\mathring{\beta}^{(1)}$  & .5 & 4 & 2.3 \\
N.1.7 & 100 & 3000 & $\mathring{\beta}^{(1)}$  & .7 & 4 & 2.6 \\
N.1.9 & 100 & 3000 & $\mathring{\beta}^{(1)}$  & .9 & 4 & 3 \\
%\hline
%1d & 100 & 3000 & $\mathring{\beta}^{(1)}$  & equi(.5) & 4 & 2.7  \\
%1e & 100 & 3000 & $\mathring{\beta}^{(1)}$  & equi(.7) & 4 & 2.9  \\
%1e & 100 & 3000 & $\mathring{\beta}$^{(1)}  & equi(.9) & 4 &  3.1 \\
\hline
N.2.5 & 200 & 2000 & $\mathring{\beta}^{(2)}$  & .5 & 7 & 2.4 \\
N.2.7 & 200 & 2000 & $\mathring{\beta}^{(2)}$  & .7 & 7 & 2.3 \\
N.2.9 & 200 & 2000 & $\mathring{\beta}^{(2)}$  & .9 & 7 & 2.2 \\
%\hline
%2d & 200 & 2000 & $\mathring{\beta}^{(2)}$  & equi(.5) & 7 & 2.3 \\
%2e & 200 & 2000 & $\mathring{\beta}^{(2)}$  & equi(.7) & 7 & 2.2 \\
%2f & 200 & 2000 & $\mathring{\beta}^{(2)}$  & equi(.9) & 7 & 2.1 \\
 \hline
\end{tabular}
\end{small}
\end{table}
For every experiment the results were based on $N=1000$ simulation runs. 

The output grid was chosen as $\lambda_k^2= c_k \cdot  \log(p) \cdot \sigma^2, k=1,\ldots,r,$ where $(c_k)_{k=1}^r = (.25,.5, \ldots, 7.5).$  We reported the mean model dimension (MD) that is 
$|\widehat{T_k}|, k=1,\ldots,r$ and 
%the percent of times the true model was correctly identified (denoted by TM), selection error $SE = (1000-TM)/1000$, 
the mean squared  prediction error (PE) on new data set with 1000 observations equalling ${\|X\mathring{\beta} - X\widehat{\beta}_{\widehat{T_k}} \|^2}/({n \sigma^2})$, where  
$\widehat{\beta}_{\widehat{T_k}}$ is the post-selection OLS estimator. Values of $(MD(k),PE(k))$ for the models chosen by GIC=GIC($\lambda_k$) were calculated 
and averaged over simulations.

The results are presented in two first columns of Figure~\ref{wyn_1}. 
% For every simulation setup and algorithm, 
%a path of models using 30 values of $c = .25,.5, \ldots, 7.5$ was chosen. Mean model dimensions (MD) and PE are presented in 
%columns 1--2 of Figure \ref{wyn_1}. 
The two vertical lines indicate models chosen using GIC with $c_k = 2.5$: the black one for SSnet and the red one for 
{\tt SparseNet}. The  blue vertical line denotes the true model dimension.
\begin{figure}[htb]
\caption{Results for simulated data}\label{wyn_1}
\centerline{%
\includegraphics[width=120mm]{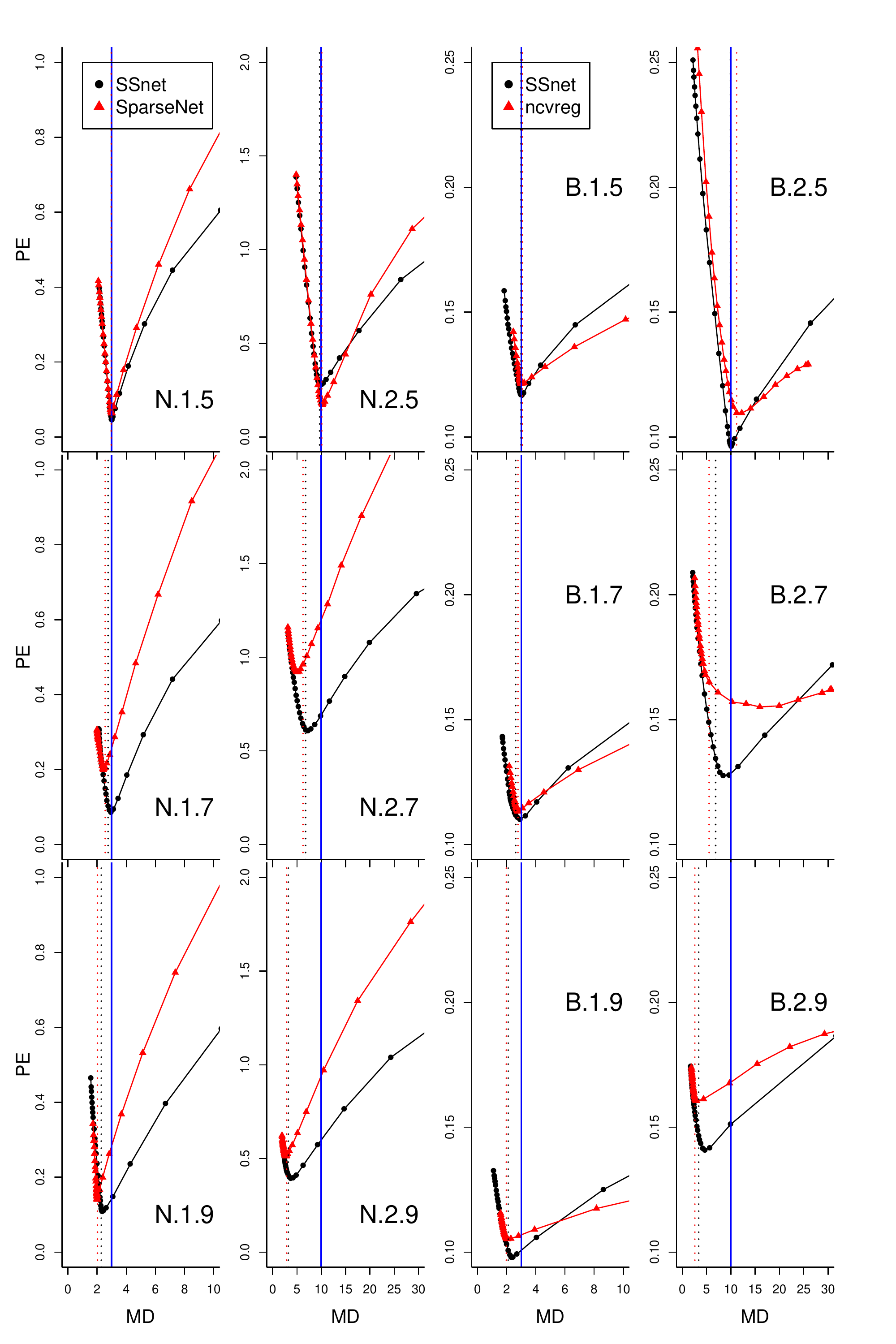}}%[width=145mm]{rys_pred_maly.pdf}}
\end{figure}

For logistic models we compared the performance of two algorithms: SSnet and MCP implemented in the R package {\tt ncvreg}  for the default  value 
of  $\gamma = 3$
and 100 values of $\lambda$. As for linear models, SSnet called the R package {\tt glmnet} \cite{FriedmanEtAl10} 
to compute the Lasso estimators for 20 lambdas on a default log scale.
%The maximal value of  $\lambda$s was $|2*X^Ty/n|_{\infty}$, the minimal was .0001 times smaller 
%and the rest were evenly spaced on a log scale giving the total of 20 values, the SOSnet parameter equaled $o = 5$. 
We performed experiments very similar to those for linear models, changing only $n$ and the number of simulation 
runs to $N = 500$. The plan of experiments is shown in the Table \ref{T2}. Random samples were generated according to the binomial distribution. 
%$y_i \sim b(1,exp(x_{i.}^T\mathring{\beta})/(1+exp(x_{i.}^T\mathring{\beta}))$.
We reported 
%the percent of times the true model was correctly identified (denoted by TM), selection error $SE = (500-TM)/500$, 
prediction error defined as misclassification frequency  
on new data set with 1000 observations. The results organized in a similar way 
as for the linear models are shown 
%for models chosen by GIC with the same 30 values of $c$ are 
in columns 3--4 of Figure \ref{wyn_1}. The two vertical lines indicate models chosen using GIC with $c_k = 2$, the black one for SSnet 
and the red one for {\tt ncvreg}.
%The mean execution time of SOSnet divided by the mean execution time of {\tt cvplogistic} was equal to 
%0.05, 0.02, 0.08, 0.15, 0.21, 0.29 for simulation scenarios: $B.1.5, B.1.7, \ldots, B.2.9$, respectively. 

Summarizing the results of the simulation study, one can observe that SSnet for linear models turned out to have equal or lower PE in almost 
all of the experimental setups.  The differences are  most visible in setups with autocorrelation structure with $\rho=0.7$. 
The value $c_k = 2.5$ in GIC usually gave  satisfactory results.
The mean execution time of SSnet was approximately 3 times faster than for {\tt SparseNet}.
SSnet for logistic regression gave higher accuracy as {\tt ncvreg}, but 
execution time of SSnet was approximately 10 times longer than {\tt ncvreg}.
The value $c_k = 2$ in GIC usually gave  satisfactory results.
%In analogous experiments we also considered equi-correlation structure of covariance matrix that is 
%$\Xi = \big(\rho + 0.1\cdot \mathbb{I}(i = j)\big)_{i,j = 1}^p$ for $\rho=0.5,0.7,0.9$. For such  cases, however, results showed that SOSnet 
%and its competitor had similar prediction errors in almost all the cases for both types of models. 
\begin{table}[tb]%%[H]
\centering
\caption{Plan of experiments for logistic models.}
\label{T2}
\begin{small}
\begin{tabular}{lcccccccccc}
  \hline
& n & p & $\mathring{\beta}$ & $\rho$  \\
\hline
B.1.5 & 300 & 3000 & $\mathring{\beta}^{(1)}$  & .5 \\
B.1.7 & 300 & 3000 & $\mathring{\beta}^{(1)}$  & .7  \\
B.1.9 & 300 & 3000 & $\mathring{\beta}^{(1)}$  & .9\\
%\hline
%1d & 300 & 3000 & $\mathring{\beta}^{(1)}$  & equi(.5)   \\
%1e & 300 & 3000 & $\mathring{\beta}^{(1)}$  & equi(.7)   \\
%1e & 300 & 3000 & $\mathring{\beta}^{(1)}$  & equi(.9) \\
\hline
B.2.5 & 500 & 2000 & $\mathring{\beta}^{(2)}$  & .5  \\
B.2.7 & 500 & 2000 & $\mathring{\beta}^{(2)}$  & .7 \\
B.2.9 & 500 & 2000 & $\mathring{\beta}^{(2)}$  & .9 \\
%\hline
%2d & 500 & 2000 & $\mathring{\beta}^{(2)}$  & equi(.5)  \\
%2e & 500 & 2000 & $\mathring{\beta}^{(2)}$  & equi(.7)  \\
%2f & 500 & 2000 & $\mathring{\beta}^{(2)}$  & equi(.9)  \\
 \hline
\end{tabular}
\end{small}
\end{table}

\subsection{Real Data Sets}

The methylation data set was described in \cite{HannumEtAl13}. It consist of the age of  656 human individuals  together with values of 
phenotypic features such as gender and body mass index and of genetic features, which are methylation states of 485~577 CpG markers. 
Methylation was recorded as a fraction representing the frequency of methylation of a given CpG marker 
across the population of blood cells taken from a single individual. In our comparison we used only genetic features from which
we extracted 193~870 most relevant CpGs according to onefold t-tests with Benjamini-Hochberg adjustment, FDR=.05. 
We compared the root mean squared errors (PE) and model dimensions (MD) for SSnet and {\tt SparseNet}. To calculate PE we used  10-fold cross-validation.
For each value of hyperparameter  $c_k = .25,.5, \ldots, 7.5$  values of $(MD(k),PE(k))$ for the models chosen by GIC=GIC($\lambda_k$) were calculated and averaged over 10 folds.
%We chose the model by GIC with $\sigma^2$ estimated using $n/2$ dimensional lasso fit. 
The results are presented in Figure \ref{wyn_meth1}. 
{\tt SparseNet} yields a path of models for each value of parameter $\gamma=g1,\ldots,g9$. 
We present
results for $g1$, corresponding to the Lasso,  for $g9$, close to the best subset, and   for  an intermediate value $g8$ in Figure \ref{wyn_meth1}.
%For {\tt SparseNet} there are additional results for models for single $\gamma$ values, where $\lambda$ 
%was chosen by GIC: $g1$ close to Lasso, $g9$ close to the best subset and $g8$ in between. 
Remarkably, SSnet gives  uniformly  smaller PE than  
{\tt SparseNet} for all $MD\geq 3$.  The two vertical lines indicate models chosen using GIC with $c_k = 2.5$: 
the black one for SSnet and the red one for {\tt SparseNet}. 
\begin{figure}[hbt]
\label{FigureSmile}
\caption{Results for the methylation data set}\label{wyn_meth1}
\centerline{%
\includegraphics[width=70mm]{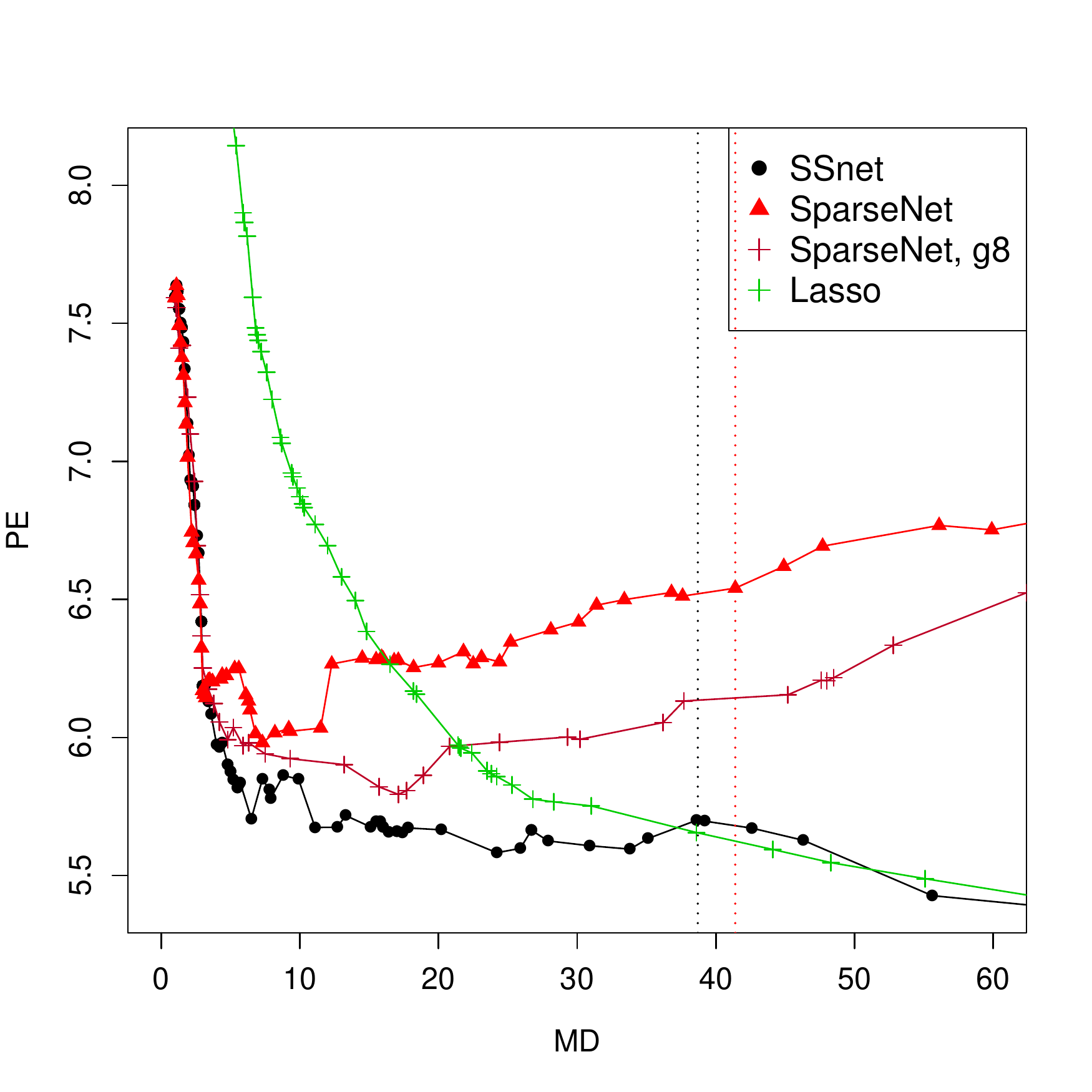}}
\end{figure}

A logistic model was fitted to the breast cancer data  described in \cite{GravierEtAl10} which concerns small, invasive  carcinomas without 
axillary lymph node involvement to predict metastasis of small node-negative breast carcinoma. There were 168 patients: 111 with no event 
after diagnosis labeled as good, and 57  with early metastasis  labeled as poor. The number of predictors in this data was 2905. 
We compared the mean errors of binary prediction (PE) and model dimensions (MD) for SSnet and {\tt ncvreg}. 
For each  of 80 hyperparameters
$c_k = 0.001; 0.005; 0.01; 0.02, \ldots,$ $0.1, 0.15,\ldots,1, \ldots 50$ values of $(MD(k); PE(k))$ for the models chosen by
$GIC=GIC(\lambda_k)$ were calculated and averaged over 10 folds. Again to calculate PE we used 10-fold cross-validation.
 The results are presented in Figure \ref{wyn_breast}. 
It is hard to find the minimal of PE for  {\tt ncvreg}. 
If we increased the net of $c_k$, maybe we would obtain a smaller PE for ncvreg, but the model would be significantly larger than for  SSnet. The algorithms work comparably, but again SSnet was 3 times longer. The two vertical lines indicate models chosen using GIC 
with $c_k = 2$: the black one for SSnet and the red one for {\tt ncvreg}.
\begin{figure}[h]
\caption{Results for the breast cancer data set}\label{wyn_breast}
\centerline{%
\includegraphics[width=70mm]{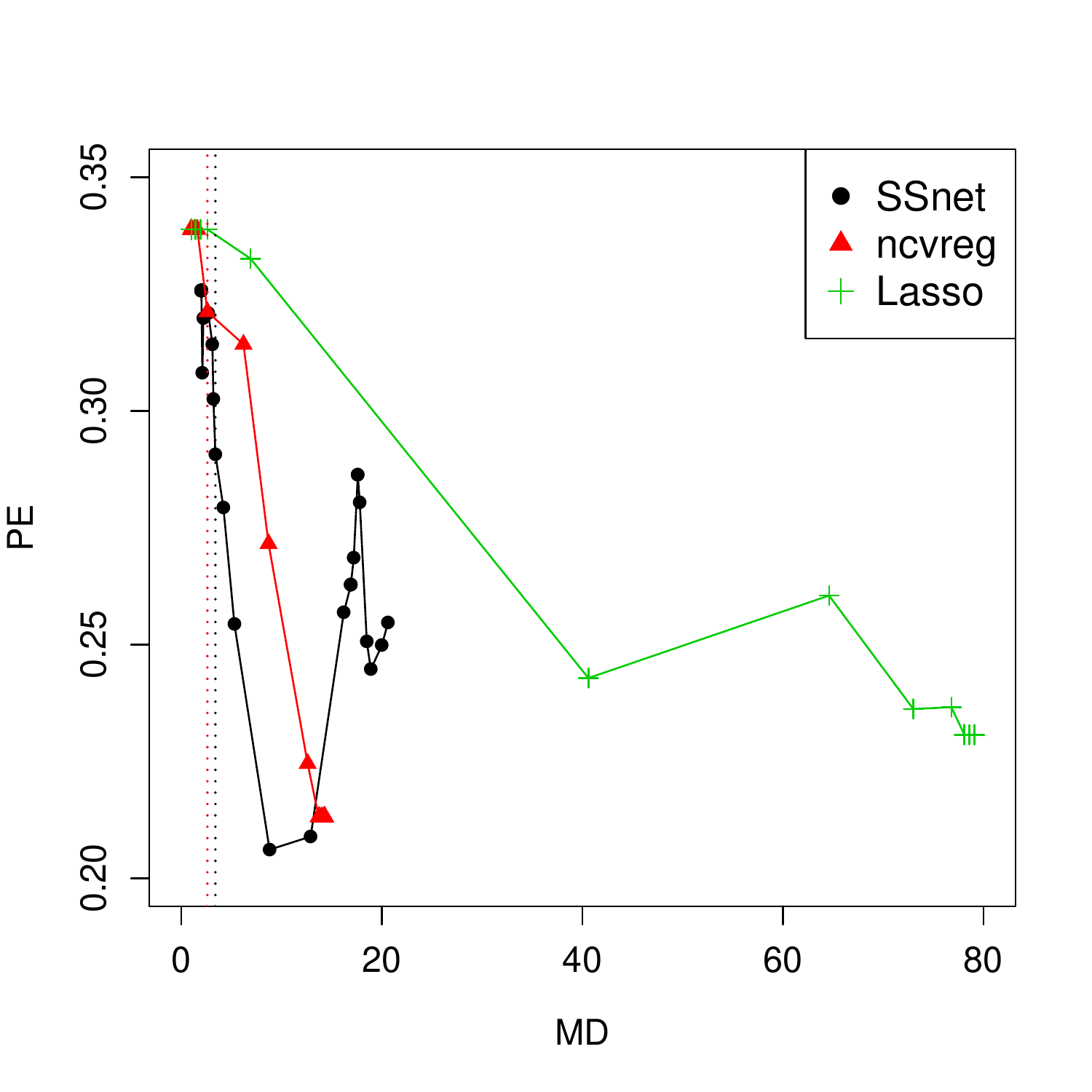}}%[width=145mm]{rys_pred_maly.pdf}}
\end{figure}

\section{Conclusions}
\label{conclusions}

In the paper we propose the SS algorithm  which is an  alternative method to TL and FCP of improving the Lasso. 
Our approach encompasses fundamental models for prediction of continuous as well as of binary responses 
and the main results are stated jointly for both of them. Its assumptions are stated in the most general form which 
allows proving exponential bound without obscuring the essence of the results and
comparing the bounds for both models. By simplifying SOS to SS we were able to simplify
reasoning used for SOS and then extend them from normal linear models to general predictive models.
%We discussed in the introduction that the Lasso is asymptotically selection consistent under very restrictive
%conditions and we propose Lasso-based method which does not suffer from this drawback.
%We show that for linear models $\lambda^2$ given in (\ref{star1}), which is the safest choice for prediction
%and estimation via Lasso or SS, it is also the safest for selection via SS. 

We propose an algorithm SSnet, which is a generalization of the SS algorithm for general predictive models.
Using a net of parameters, SSnet avoids the problem of choosing one specific $\lambda$.
The gap between theoretical results for SS and the SSnet algorithm is similar to the difference between theory for FCP and it implementations {\tt SparseNet} or {\tt ncvreg}.
Numerical experiments reveal that  for normal linear models  SSnet is more accurate than {\tt SparseNet} and three times faster, 
whereas for logistic models performance of SSnet is also better than
the performance of {\tt ncvreg} with computing times 3-10 times longer.

We have shown in simulations (dotted vertical lines in Figure~1) that predictively optimal $\lambda$ for normal 
linear models equals approximately $\sqrt{2.5\sigma^2\log p}$, which is close to (\ref{star1}) and for 
logistic models is $\sqrt{2\log p}$, which together with (\ref{star2})  suggests that $c \approx 1/4$. 
The relations between the safest choice $\lambda$ discussed in Remark \ref{safest_choice} and predictively optimal $\lambda$ are important applications of our theory.

In the package ,,DMRnet'' one can find the SOSnet algorithm, which is a net version of SOS from \cite{PokarowskiMielniczuk15}. SOSnet can be viewed as SSnet with additional step (,,Ordering'') based on refitting estimators and calculating Wald statistics. Results of our numerical experiments, which are not presented here, suggest that SOSnet can improve the quality of SSnet, especially in the linear model.

\appendix

\section{Proofs and auxiliary results
}
In the following subsections  we present proofs of  the results stated in the paper.

\subsection{Proof of Lemma~\ref{lemSubG}}
%\begin{proof}[Proof of Lemma~\ref{lemSubG}]
Let $Z=\varepsilon^T\nu/||\nu||$ and $a>0$. From Markov's inequality we obtain 
\[ \mathbb{P}(Z \geq \tau) \leq e^{-a}\mathbb{E}e^{aZ} \leq e^{-a\tau + a^2\sigma^2/2}.\]
Minimizing the last expression w.r.t. $a$ gives part (i).

Let $\xi \sim N(0,I_n)$ be a vector of iid standard normal errors independent of $\varepsilon$. We have
\begin{eqnarray*}
& & \mathbb{E}\exp\left(\frac{a}{2\sigma^2}\varepsilon^T H\varepsilon\right) = \mathbb{E}~\mathbb{E} \bigg[ \exp\left(\frac{\sqrt{a}}{\sigma}\xi^T H\varepsilon \right) \bigg| H\varepsilon \bigg]
= \mathbb{E} \exp\bigg(\frac{\sqrt{a}}{\sigma}\xi^T H\varepsilon\bigg) \cr & &= \mathbb{E}~\mathbb{E} \bigg[ \exp\left(\frac{\sqrt{a}}{\sigma}\xi^T H\varepsilon \right) \bigg| \xi^TH \bigg] 
\leq \mathbb{E}\exp\left(\frac{a}{2}\xi^T H\xi\right).
\end{eqnarray*}
Thus part (ii) follows  from a known formula for the moment generating function of the $\chi^2$ distribution.

From Markov's inequality and the part (ii) of this lemma we have
\begin{eqnarray*} 
\mathbb{P}(\varepsilon^T H\varepsilon \geq m\sigma^2\tau) 
&\leq& \exp \bigg( -\frac{am\tau}{2} \bigg) \mathbb{E}\exp \bigg( \frac{a}{2\sigma^2}\varepsilon^T H\varepsilon\bigg) \\
 &\leq& \exp \bigg( -\frac{m}{2}\bigg(a\tau+\log(1-a)\bigg)\bigg). 
\end{eqnarray*}
Thus, minimizing the last expression  w.r.t $a$ we obtain part (iii).
\qed
%\end{proof}
%\hfill\BlackBox \\
%%%%%%%%%%%%%

\subsection{Proof of Lemma~\ref{LassoBound}}
%\begin{proof}[Proof of Lemma~\ref{LassoBound}]
Let  
\begin{equation}
\label{Aevent}
{\cal A}_{a}=\{|X^T\varepsilon|_\infty\leq a\lambda\}
\end{equation}
 and $\hat\varepsilon=y-\vec{\dot\gamma}(X\hat{\beta})$. We have 
$\dot \ell(\hat{\beta}) = -X^T\hat{\varepsilon}$ and from the Karush-Kuhn-Tucker (KKT) theorem we obtain equations 
 \begin{equation*}
X^T\hat{\varepsilon} = \lambda \big[ \mathbb{I}(\hat{\beta}>0)-\mathbb{I}(\hat{\beta}<0)+u\mathbb{I}(\hat{\beta}=0) \big]~~\text{for}~~ u \in [-1,1]. 
 \end{equation*}
 
Let $\Delta=\hat\beta-\mathring{\beta}$ and $\nu\in \mathbb{R}^p$ be  such that ${\rm sgn}(\nu_{\bar{T}})={\rm sgn}(\Delta_{\bar{T}})$. 
We have $\nu_J^TX_J^T\hat\varepsilon = \lambda|\nu_J|_1$ for $J \subseteq \bar{T}$ and consequently
 \begin{align}
  &D(\nu)= \nu^TX^T\big[\vec{\dot\gamma}(X\hat{\beta}) -\vec{\dot\gamma}(X\mathring{\beta})\big] 
  = \nu_T^TX_T^T(\varepsilon-\hat{\varepsilon}) +\nu_{\bar T}^TX_{\bar T}^T(\varepsilon-\hat{\varepsilon)} \cr
  &\leq |\nu_T|_1(|X_T^T\varepsilon|_\infty + |X_T^T\hat{\varepsilon}|_\infty)+  
    |\nu_{\bar T}|_1(|X_{\bar T}^T\varepsilon|_\infty-\lambda) \cr 
 & \leq  |\nu_T|_1 (1+ a)\lambda+ |\nu_{\bar T}|_1(a-1)\lambda.
 \end{align}
 Then letting $\nu=\Delta_J$  for $J\subseteq \bar T$ we have $D(\nu) \leq 0$. Moreover, for $\nu=\Delta$ we have 
from convexity of $g$ that
\[ D(\nu) = (\hat{\beta}-\mathring{\beta})^T[\dot g(\hat{\beta})-\dot g(\mathring{\beta})]\geq 0.\]
Indeed, $D_0(\beta_1,\beta_2)=(\beta_1-\beta_2)^T[\dot g(\beta_1)-\dot g(\beta_2)]$ is the symmetrized Bregman divergence
\citep[][]{HuangZhang12}. Hence $(1-a)|\nu_{\bar T}|_1 \leq (1+ a)|\nu_T|_1$. 
Thus, on ${\cal A}_a$, $\Delta \in {\cal C}_a$  and from the definition of $\zeta_a$  we obtain using KKT again
\begin{equation*} 
\zeta_{a,q}|\Delta|_q 
    \leq \big|X^T\big[\vec{\dot\gamma}(X\hat{\beta})) -\vec{\dot\gamma}(X\mathring{\beta})\big]\big|_\infty
   \leq |X^T\hat\varepsilon|_\infty  + |X^T\varepsilon|_\infty \leq (1+a)\lambda.
\end{equation*}
\qed
%\end{proof}
%%%%%%%%%%%%%
%\subsection{A Bound on the Selection Error of TL}
\subsection{Proof of Theorem~\ref{th1}}
%Proof of theorem \ref{th1}.
%\begin{proof}[Proof of Theorem~\ref{th1}]
First, we will prove that ${\cal A}_a \subseteq \{\hat{T}_{TL} = T \}$ for ${\cal A} _a$ defined in \eqref{Aevent}.
From Lemma~\ref{LassoBound} and assumptions  we have on ${\cal A}_{a}$ 
 \begin{equation}
 \label{separ}
 |\Delta|_\infty\leq (1+a)\lambda \zeta_a^{-1} \leq \tau < \mathring{\beta}_{min}/2,
 \end{equation}
where we recall that $\Delta=\hat\beta-\mathring{\beta}.$
 Thus using (\ref{separ}) twice we have for $j\in T$ and $k\not\in T$
\begin{equation}
\label{rozdziel}
 |\hat\beta_j| \geq |\mathring{\beta}_j|-|\hat\beta_j-\mathring{\beta}_j| > \mathring{\beta}_{min}-\mathring{\beta}_{min}/2 
 > \tau \geq |\hat{\beta}_k -\mathring{\beta}_k|=|\hat \beta_k|
\end{equation}
and it follows that ${\cal A}_a \subseteq \{\hat{T}_{TL} = T \}$.
Morover, the assumptions of  this theorem imply
%Ponadto, z zalozenia wniosku wynika, ze
 \[
-a_1^2\lambda^2+2\sigma^2\log p \leq -(1-a_2)a_1^2\lambda^2.
\]  
Hence, using Lemma~\ref{lemSubG}~(i) we easily obtain
 \begin{equation*}
 %\label{setA}
  \mathbb{P}(\hat{T}_{TL} \neq T) \leq \mathbb{P}({\cal A}_{a_1}^c)\leq 2p\exp\big(-\frac{a_1^2\lambda^2}{2\sigma^2}\big) 
   \leq 2\exp\big(-\frac{(1-a_2)a_1^2\lambda^2}{2\sigma^2}\big).
  \end{equation*}
\qed
%\end{proof}
%\hfill\BlackBox \\

%%%%%%%%%%%%%
%\subsection{A Bound on the Selection Error of SS}
\subsection{Proof of Theorem~\ref{th2}}

%In the following subsection we will prove theorem \ref{th2}, that is we will estimate selection error of SS algorithm. 
%\begin{proof}[Proof of Theorem~\ref{th2}] 
Let us observe that the consecutive steps of the SS algorithm constitute a decomposition of the selection error into two parts:
\begin{equation}
\label{error_decom}
\{ \hat{T} \neq T\} = \{T\not\in {\cal J}\} \cup \{T\in {\cal J},\hat{T}\neq T \}.
\end{equation}
Therefore, Theorem~\ref{th2} follows easily from (\ref{errLasso2}) and (\ref{errGIC}) below.

Having in mind that for given $a_1 \in (1/2,1)$ we let $a_2=1-(1-\log (1-a_1))(1-a_1)$, $a_3=2-1/a_1$ and $a_4=\sqrt{a_1a_2}$,
by arguments similar to those in the proof of Theorem~\ref{th1} we obtain ${\cal A}_{a_4} \subseteq \{ T \in {\cal J} \}$.
Moreover, assumptions $0<c\leq 1$ and $\frac{2\sigma^2\log p}{a_3a_2a_1c} \leq \lambda^2 $ imply
\begin{equation*}
-a_4^2\lambda^2+2\sigma^2\log p \leq -a_2a_1c\lambda^2 +2\sigma^2\log p 
 \leq -(1-a_3)a_2a_1c\lambda^2 =  -a_2(1-a_1)c\lambda^2.
\end{equation*}
As a result
\begin{equation}
\label{errLasso2}
 \mathbb{P}\big( T\not\in {\cal J}  \big) \leq \mathbb{P}({\cal A}_{a_4}^c) \leq 2p\exp\big(-\frac{a_4^2\lambda^2}{2\sigma^2}\big) 
  \leq 2\exp\big(-\frac{a_2(1-a_1)c\lambda^2}{2\sigma^2}\big).
\end{equation}

%%%%%%%%%%%%%%%

Now we bound  $\mathbb{P}\big( T\in {\cal J},\hat{T}\neq T \big)$.
In Lemma \ref{lemSub2} and \ref{lemSup}, given below, we bound probability that in the second step of the SS algorithm GIC chooses a subset of the true set, i.e.
$$
\{ T\in {\cal J},\hat{T}\subset T\}
$$
or a superset of $T,$ i.e. 
$$
\{ T\in {\cal J},\hat{T}\supset T\}.
$$
These lemmas state that both components of the selection error set are included in the critical sets of the following form
\begin{equation}
\label{quadr}
{\cal C}_J(\tau) = \{ \varepsilon^T H_J \varepsilon \geq \tau \}
\end{equation}
or 
\begin{equation}
\label{quadr2}
{\cal C}_{J\ominus T}(\tau) = \{ \varepsilon^T (H_J-H_T) \varepsilon \geq \tau \},
\end{equation}
where $J$ is such that $T \subset J,r(X_J)=|J|\leq \bar{t}.$
We consider only supersets $J$ that $r(X_J)=|J|,$ because GIC corresponding to the superset $J$ such that $r(X_J)<|J|$ is larger than GIC 
corresponding to a superset $J_1$ such that $J_1 
\subsetneq J$ and $r(X_{J_1})=|J_1|.$

Let us define $\tau_0=\frac{1}{1-a_1},~\tau_1=\frac{(1-a_1)c\lambda^2}{t\sigma^2}$ and $\tau_2=\frac{a_1c\lambda^2}{\sigma^2}$. 
Under our assumptions we have $2<\tau_0<\tau_1<\tau_2$. Let $f_2(\tau)=1-(1+\log \tau)/\tau$ for $\tau>1$.
Of course $f_2$ is increasing, $f_2(1)=0$ and $f_2(\tau)\to 1$ for $\tau\to\infty$. Consequently $a_2=f_2(\tau_0)<f_2(\tau_1)<f_2(\tau_2)$, which means that
\begin{equation}
\label{tau}
a_2\tau_r<\tau_r-1-\log \tau_r~\text{for}~r=1,2.
\end{equation}
From Lemma~\ref{lemSub2}, Lemma~\ref{lemSup}, Lemma~\ref{lemSubG}~(iii) and (\ref{tau}) we get 
\begin{align*}
 & \mathbb{P}\big( T\in {\cal J},\hat{T} \neq T  \big)  
 \leq \mathbb{P}\big( {\cal C}_T(t\sigma^2\tau_1) \big)
    + \sum_{J\supset T:r(X_J)=|J|\leq \bar{t}}  \mathbb{P}\big(  {\cal C}_{J\ominus T}(|J\setminus T|\sigma^2\tau_2) \big) \cr
& \leq \exp\big(-ta_2\tau_1/2\big) + \sum_{m=1}^{p-t} {p-t\choose m} \exp\big(-ma_2\tau_2/2\big)
 \cr & \leq \exp\big(-ta_2\tau_1/2\big) + \sum_{m=1}^{p-t} \frac{1}{m!} \exp \left(\frac{-m}{2}(a_2\tau_2-2\log p)\right).
\end{align*}
Using $\exp(d)-1 \leq \log(2)^{-1}d$ for $0\leq d \leq \log(2)$ and the fact that
probability is not greater than $1$ we obtain
\begin{equation*}
\mathbb{P}\big( T\in {\cal J},\hat{T} \neq T  \big) 
 \leq \exp\big(-ta_2\tau_1/2\big) + (\log 2)^{-1} \exp\big(-(a_2\tau_2-2\log p)/2\big). 
\end{equation*}
For $a_3 \in (0,1)$, assumption $ \frac{2\sigma^2\log p}{a_3a_2a_1c} \leq \lambda^2$ implies 
$-a_2\tau_2+2\log p \leq -(1-a_3)a_2\tau_2$, therefore
\begin{equation}
\label{errGIC}
 \mathbb{P}\big( T\in {\cal J},\hat{T}\neq T  \big) \leq  \exp\big(-ta_2\tau_1/2\big) 
  +   (\log 2)^{-1}\exp\big(-(1-a_3)a_2\tau_2/2\big),
\end{equation}
%\end{proof} % proof of theorem 2
that finishes the proof of Theorem \ref{th2}. \qed 

Before we state Lemma \ref{lemSub2} and \ref{lemSup} we introduce few notations.
For $k=1,\ldots,t-1$ and $\delta_k$ in \eqref{delta_k} we define 
\begin{eqnarray}
\label{calE}
{\cal E}_k(\tau) &=& \{\exists J\subset T, \, |T\setminus J|=k:\, \ell_J -\ell_T \leq \tau \}, \\
 B_k&=&\{\beta_T: ||X_T\beta_T-X_T\mathring{\beta}_T||^2 \leq \delta_k \}. \nonumber
\end{eqnarray}
First, we prove the following lemma which will be used in the proof of Lemma \ref{lemSub2}.

\begin{lemma}
\label{aux_lemma}
 For $b \in (0,1)$  we have 
\[ {\cal E}_k(bc\delta_k/2) \subseteq  {\cal C}_T((1-b)^2c^2\delta_k/4). \]
\end{lemma}

%Proof of lemma \ref{lemSub1}.
\begin{proof}
For $\beta_T \in \partial B_k$  from assumption \eqref{convexLoss}, Schwartz inequality and properties of the 
orthogonal projection $H_T$  we get
\begin{eqnarray*}
\mathring{\ell}(\beta_T) := \ell(\beta_T) - \ell(\mathring{\beta}_T) &\geq&  (\mathring{\beta}_T - \beta_T)^T X_T^TH_T\varepsilon 
 + \frac{c}{2} (\mathring{\beta}_T - \beta_T)^T X_T^TX_T(\mathring{\beta}_T - \beta_T)
\\ & \geq& - \sqrt{\delta_k\varepsilon^T H_T\varepsilon} + \frac{c}{2}\delta_k.
\end{eqnarray*}
Since the last expression does not depend on $\beta_T$, we have for $b \in (0,1)$
\begin{eqnarray*}
{\cal L}_k(b) &:=& \left\{ \min_{\beta_T \in \partial B_k} \mathring{\ell}(\beta_T) \leq \frac{bc\delta_k}{2} \right\} 
 \subseteq \left\{ -\sqrt{\delta_k\varepsilon^T H_T \varepsilon} + \frac{c}{2}\delta_k \leq \frac{bc\delta_k}{2} \right\} \\
&=& {\cal C}_T \left( \frac{(1-b)^2c^2 \delta_k}{4} \right).
\end{eqnarray*}
Let us notice that for $ J\subset T$ such that $|T\setminus J|=k$ we have 
$||X_T(\widehat{\beta}_J^{ML} - \mathring{\beta}_T) ||^2 \geq \delta_k$, 
so $\widehat{\beta}_J^{ML} \notin \text{int}(B_k)$. 
Since $\mathring{\ell}$ is convex and $\mathring{\ell}(\mathring{\beta}_T) = 0$ we obtain
$ {\cal L}_k(b) \supseteq {\cal E}_k(bc\delta_k/2)$.
\end{proof}
%\hfill\BlackBox \\

\begin{lemma}
\label{lemSub2}
If for $a \in (0,1)$ we have $\lambda^2 < c\delta\big/\big(1+\sqrt{2a}\big)^2 $, then
$\{T\in {\cal J},\hat{T}\subset T \} \subseteq {\cal C}_T(ac\lambda^2).$
\end{lemma}

\begin{proof}
From assumption of this lemma $\lambda^2<c\delta$, so $b_k: =\lambda^2k/(c\delta_k)<1$. Hence
for $k=1,\ldots,t-1$
\begin{equation}
\label{sub1}
k\lambda^2 \leq b_kc\delta_k.
\end{equation}
Moreover, if $a\in (0,1)$ then
\begin{equation}
\label{sub2}
akc\lambda^2 \leq (1-b_k)^2c^2\delta_k/4,
\end{equation}
because $ab \leq (1-b)^2/4$ for $b=\lambda^2/(c\delta)$.
The last inequality is true, if
\begin{equation}
\label{quadratic}
b \leq 1 +2a-\sqrt{(1+2a)^2-1}=f_1(a).
\end{equation}
Indeed, it is easy to check that (\ref{quadratic}) follows from the assumption as
\[f_1(a)=\frac{1}{1+2a +\sqrt{(1+2a)^2-1}}\geq \frac{1}{(1+\sqrt{2a})^2}.\]
Finally from (\ref{sub1}), Lemma \ref{aux_lemma} and (\ref{sub2}) we obtain, respectively
\begin{eqnarray*}
\{T\in {\cal J},\hat{T}\subset T \} &\subseteq& \bigcup_{k=1}^{t-1}{\cal E}_k \left(\frac{k\lambda^2}{2}\right)
\subseteq \bigcup_{k=1}^{t-1}{\cal E}_k \left(\frac{b_kc\delta_k}{2} \right) 
\\ &\subseteq& \bigcup_{k=1}^{t-1}{\cal C}_T \left(\frac{(1-b_k)^2c^2\delta_k}{4} \right) \subseteq {\cal C}_T(ac\lambda^2)
\end{eqnarray*}
\end{proof}
%\hfill\BlackBox \\

\begin{lemma}
\label{lemSup}
For $a \in (0,1)$ we have  
\begin{equation}
\label{supers}
 \{T\in {\cal J},\hat{T}\supset T \} \subseteq {\cal C}_T((1-a)c\lambda^2)        
 \cup \bigcup_{J \supset T:r(X_J)=|J|\leq \bar{t}} {\cal C}_{J\ominus T}(|J\setminus T|ac\lambda^2). 
\end{equation}
\end{lemma}

\begin{proof}
For $J$ such that $J \supset T,r(X_J)=|J|\leq \bar{t}$ define $W_J=X_J(\mathring{\beta}_J-\hat{\beta}_J^{ML})$ and $m=|J\setminus T|$. 
Notice that the event $\{\ell_J - \ell_T \leq -m\lambda^2/2 \}$
can be decomposed as 
\begin{equation}
\label{decomp}
\{\ell_J - \ell_T \leq -m\lambda^2/2, \hat{\beta}_J^{ML} \in  B_J \} \cup\{\ell_J - \ell_T \leq -m\lambda^2/2 , \hat{\beta}_J^{ML} \notin  B_J\},
\end{equation}
where $
B_J =\{\beta_J: ||X_J (\bo-\beta_J)||^2 \leq \dt\} .
$
For $\hat{\beta}_J^{ML} \in  B_J$ we have
from assumption \eqref{convexLoss}  and properties of orthogonal projection $H_J$
\begin{eqnarray*}
\ell_J - \ell_T &\geq& \ell(\hat{\beta}_J^{ML}) - \ell(\mathring{\beta}_J) \\
 &\geq& (\mathring{\beta}_J-\hat{\beta}_J^{ML})^TX_J^TH_J\varepsilon 
 + \frac{c}{2}(\mathring{\beta}_J-\hat{\beta}_J^{ML})^T X_J^TX_J (\mathring{\beta}_J-\hat{\beta}_J^{ML})
\\ &=& W_J^TH_J\varepsilon + \frac{c}{2}W_J^TW_J = \frac{1}{2c}||cW_J+H_J\varepsilon ||^2 - \frac{1}{2c}\varepsilon^TH_J\varepsilon \\
 &\geq&  - \frac{1}{2c}\varepsilon^TH_J\varepsilon,
\end{eqnarray*}
so $\{\ell_J - \ell_T \leq -m\lambda^2/2, \hat{\beta}_J^{ML} \in  B_J \} \subseteq \{ \varepsilon^TH_J\varepsilon \geq mc\lambda^2\}$. 
The second event in the sum \eqref{decomp} is obviously contained in $\{\hat{\beta}_J^{ML} \notin  B_J\}.$ 
Now, we show that this event is also contained in $\{\varepsilon^TH_J\varepsilon \geq mc\lambda^2\}$.
To do it, we use argumentation similar to \citet{geer02}. Thus, we define
\begin{eqnarray*}
&\,&d_J=||X_J (\hbj ^{ML} -\b*)||, \quad\quad \quad u=\frac{\sqrt{\delta_{t-1}}}{\sqrt{\delta_{t-1}}+d_J} ,\\
&\,& \tbj =u \hbj ^{ML}+(1-u)\b* , \quad\quad \quad \tilde{d}_J=||X_J (\tbj -\b*)||.
\end{eqnarray*}
The random vector $\tbj$ belongs to $B_J,$ because
$$
 ||X_J(\tbj  -\b*) ||=u d_J \leq \sqrt{\delta_{t-1}}.
$$
Using convexity of the loss function $\ell$ we have 
$$
\ell(\tbj) \leq u \ell(\hbj ^{ML}) + (1-u) \ell(\b*)= u[\ell(\hbj ^{ML})- \ell(\b*)] + \ell(\b*)\leq \ell(\b*).
$$
Using this fact as well as assumption \eqref{convexLoss}, Schwartz inequality and properties of the 
orthogonal projection $H_J$  we get
\begin{eqnarray*}
%\label{notball1}
0 &\geq&  \ell (\tbj) - \ell(\b*)
\geq  (\mathring{\beta}_J - \tbj)^T X_J^TH_J\varepsilon 
 + \frac{c}{2} (\mathring{\beta}_J - \tbj)^T X_J^TX_J(\mathring{\beta}_J - \tbj)
\\ & \geq& - \sqrt{\tilde{d}_J^2 \varepsilon^T H_J\varepsilon} + \frac{c}{2} \tilde{d}_J^2.
\end{eqnarray*}
It gives us
\begin{equation*}
%\label{notball2}
\varepsilon^T H_J\varepsilon \geq \frac{c^2}{4} \tilde{d}_J ^2.
\end{equation*}
Therefore, we obtain 
$$
\{\hat{\beta}_J^{ML} \notin  B_J\}  = \{ \tilde{d}_J > \sqrt{\delta_{t-1}}/2\} \subset \left\{\varepsilon^TH_J\varepsilon \geq \frac{c^2 \dt }{16}
\right\}.
$$
From assumptions of Theorem \ref{th2} we know that $mc\lambda ^2 \leq \frac{  c^2 \dt}{16} \:, $
that gives us 
$\{\hat{\beta}_J^{ML} \notin  B_J\} \subset \{\varepsilon^TH_J\varepsilon \geq mc\lambda^2\}$.

On the other hand, $ \varepsilon^TH_J\varepsilon =  \varepsilon^TH_T\varepsilon +  \varepsilon^T (H_J-H_T)\varepsilon$, hence
we obtain for $\tau>0$ and $a \in (0,1)$\\
\[ \{ \varepsilon^TH_J\varepsilon \geq \tau \} \subseteq  \{ \varepsilon^TH_T\varepsilon \geq (1-a)\tau \} 
\cup  \{ \varepsilon^T(H_J-H_T)\varepsilon \geq a\tau \}. \]
Finally 
\begin{eqnarray}
\label{super_f}
&\,& \{T\in {\cal J},\hat{T}\supset T \} \subseteq 
\bigcup_{J\supset T:r(X_J)=|J| \leq \bar{t}} \{ \ell_J -\ell_T \leq  -|J\setminus T|\lambda^2/2 \} \nonumber\\
 &\,&\subseteq {\cal C}_T((1-a)c\lambda^2) \cup \bigcup_{J\supset T:r(X_J)=|J| \leq \bar{t}}  {\cal C}_{J\ominus T}(|J\setminus T|ac\lambda^2).
\end{eqnarray}

\end{proof}

\subsection{Proof of Theorem \ref{th_LM}}
\label{proofTh_LM}

The idea of the proof is similar to the proof of Theorem \ref{th2}. However, here we consider
the linear model  with the subgaussian noise that allows us to  apply simpler and better arguments than in the general case of GLM.
Indeed, while working with the event $\{T\in {\cal J}, \hat T\subset T\}$ we can use the result for the scalar product (Lemma \ref{lemSubG} (i))
instead of the one for quadratic forms (Lemma \ref{lemSubG} (ii)). Besides, when considering $\{T\in {\cal J}, \hat T \supset T\}$ we need to upper bound probability of 
$\{\varepsilon ^TH_J\varepsilon - \varepsilon ^T H_T\varepsilon \geq |J \setminus T|\lambda^2 \}$ instead of 
$\{\varepsilon ^T H_J \varepsilon \geq |J \setminus T|\lambda^2 \}.$ Obviously, probability of the former event is smaller.

As previously, we work with the error decomposition \eqref{error_decom}. Probability of the event 
$\{T\not\in {\cal J}\}$ can be bounded in the same way 
as in the proof of Theorem \ref{th2}, if we put $a_1=\sqrt{a_2}=a$.
However, working with subsets and supersets of $T$ is  different.

We start with subsets of $T.$  For $J\subset T$ we denote  $$\mu_{T\ominus J}=(H_T-H_J)X_T \b* _T$$
and 
 $$\delta_{T\ominus J}=||\mu_{T\ominus J}||^2.$$ 
Notice that 
 $\ell_J + |J|\lambda^2/2 \leq  \ell_T + t\lambda^2/2$ is equivalent to
 \[ \delta_{T\ominus J} +2\varepsilon^{T}\mu_{T\ominus J}+\varepsilon^T (H_T - H_J)\varepsilon\leq (t-|J|)\lambda^2.\]
Whence
 \begin{eqnarray*}
\bigg\{\ell_J + |J|\lambda^2/2 \leq  \ell_T + t\lambda^2/2 \bigg\}&\subseteq& \bigg\{-2\varepsilon^{T}\mu_{T\ominus J}\geq  \delta_{T\ominus J} - k\lambda^2 \bigg\} \\ &\subseteq& \bigg\{-\frac{2\varepsilon^T\mu_{T\ominus J}}{||\mu_{T\ominus J}||}\geq (k \delta)^{1/2}\bigg(1-\frac{\lambda^2}{ \delta}\bigg)\bigg\},
\end{eqnarray*}
where $k=t-|J|$ and $\delta_{T\ominus J} \geq \delta _k\geq k \delta$ by the definiton of $\delta.$
Thus, from Lemma \ref{lemSubG} (i) we have
\begin{eqnarray*}
& & P\bigg(\exists J \subset T:\, \ell_J + |J|\lambda^2/2 \leq  \ell_T + t\lambda^2/2 \bigg)\cr
& & \leq P\bigg(\exists J\subset T:\, -\frac{\varepsilon^T\mu_{T\ominus J}}{||\mu_{T\ominus J}||}\geq 2^{-1}(k \delta)^{1/2}\bigg(1-\frac{\lambda^2}{\delta}\bigg)\bigg)\cr
& &\leq \sum_{k=1}^t {t\choose k} \exp\bigg({-\frac{k}{8\sigma^2} \delta\bigg(1-\frac{\lambda^2}{ \delta}\bigg)^2}\bigg)\cr
& & \leq \sum_{k=1}^t \frac{1}{k!}\exp\bigg({-k\bigg(\frac{1}{8\sigma^2} \delta\bigg(1-\frac{\lambda^2}{ \delta}\bigg)^2-\log t \bigg)\bigg)}\cr
& &\leq  \log(2)^{-1}\exp\bigg( -\frac{1}{8\sigma^2}\delta\bigg(1-\frac{\lambda^2}{ \delta}\bigg)^2+\log t\bigg)\cr
& &\leq \log(2)^{-1}t\exp\bigg(-\frac{a^2\lambda^2}{2\sigma^2}\bigg).
\end{eqnarray*}
For the penultimate inequality we use again the inequality $\exp(c)-1 \leq log(2)^{-1}c$ for $0\leq c \leq log(2)$ and the fact that probability is not greater than $1$. For  the last inequality above we used
\[a^2\lambda^2 \leq \frac{\delta}{4}\bigg(1-\frac{\lambda^2}{\delta}\bigg)^2\]
which is satisfied for
\begin{equation}
\label{quadraticLM}
\lambda^2/\delta \leq f_1(a), 
\end{equation}
where
$$
 f_1(a)=1 +2a^2-\sqrt{(1+2a^2)^2-1}.
$$
It is easy to check that (\ref{quadraticLM}) follows from the assumption $\lambda^2 \leq \delta/(2+2a)^2$ as
\[f_1(a)=\frac{1}{1+2a^2 +\sqrt{(1+2a^2)^2-1}}\geq \frac{1}{(1+\sqrt{2}a)^2}\geq  \frac {1}{(2+2a)^2}.\]

%%%%%%%%%%%%%%%

Now, we consider supersets of $T.$ 
As previously we consider only such supersets $J$ that $r(X_J)=|J| \leq \bar t.$ We have $ \ell _T - \ell_J = \varepsilon ^T(H_J - H_T)\varepsilon/2 $.
Using Lemma~\ref{lemSubG}~(iii), as in the proof of Theorem \ref{th2}, we obtain
\begin{eqnarray*}
& & P\big(\exists J\supset T:\, \ell_J + |J|\lambda^2/2 \leq  \ell_T + t\lambda^2/2 \big) \cr
& &  = P\big(\exists J\supset T:\, \varepsilon^T(H_J -H_T )\varepsilon\geq (|J|-t)\lambda^2 \big)\cr
& & \leq \log(2)^{-1}p\exp\bigg(-\frac{a^2\lambda^2}{2\sigma^2}\bigg).
\end{eqnarray*}
Finally, we get the following inequalities 
\[
P(\hat T_{SS}\neq T) \leq \bigg(2+\frac{2}{\log(2)}\bigg)p\exp\bigg(-\frac{a^2\lambda^2}{2\sigma^2}\bigg) \leq 5\exp\bigg(-\frac{a^2(1-a^2)\lambda^2}{2\sigma^2}\bigg).
\]
%where we consider again only those supersets $J$ that $r(X_J)=|J|\leq \bar{t}.
%The rest of the proof is the same as in Theorem \ref{th2} and is based on Lemma~\ref{lemSubG}~(iii)

\qed

\subsection{Proof of Theorem \ref{main_emp}}
\label{proof_convex}

In section \ref{convex} we state Theorem \ref{main_emp}  with simplified constants $K_i.$ Its general form is given below: 

{\it
Fix $a_1,a_2 \in (0,1).$ Assume  that \eqref{strong_c}, \eqref{ZM}, \eqref{UJ} and
\begin{eqnarray}
\label{assump1}
&\quad& \max\left(\frac{4K_1^2\log (2p)}{a_1^2},\frac{32 K_3^2 \log p}{K_4c_2}, \frac{K_3^2t}{a_2c_2}, \frac{16 K_3^2(t+1)}{c_2}\right)  L^2  \leq \lambda^2 \leq  \\
\label{assump}
&\,& \leq  \min \left[
\frac{c_2 \delta}{(1+\sqrt{2a_2})^2} , \frac{c_2 \dt }{4(\bar{t} -t)}
, \frac{(1-a_1)^2 c_1^2 \kappa^2_{a_1} \bmin ^2}{16} \right].
\end{eqnarray}
Then 
\begin{equation*}
P\left( \hat{T}_{SS} \neq T 
\right) \leq 4.5 \exp\left[- \frac{\lambda^2}{L^2}\min\left( \frac{a_1^2 K_2}{4K_1^2} ,
\frac{K_4 c_2}{32K_3^2} , \frac{  a_2 K_4 c_2  }{K_3^2}
\right)\right].
\end{equation*}
}
The main difference between proofs of Theorems \ref{th2} and \ref{main_emp} is  that here we 
investigate properties of the expected loss $\Ex \ell (\cdot)$ instead of the loss $\ell (\cdot).$
It relates to the fact that the former function is more regular. To be more precise, in many parts of the proof we work with expressions of the form
\begin{equation}
\label{dec_emp}
\ell(\tilde{\beta})- \ell(\b*),
\end{equation}
where $\tilde{\beta}$ is a random vector contained in some ball $\tilde{B}.$
Clearly, we can transform \eqref{dec_emp} into 
\begin{equation}
\label{dec_emp1}
\left[ \Ex \ell(\tilde{\beta})- \Ex \ell(\b*)\right] + \left[ \ell(\tilde{\beta})- \ell(\b*) - \Ex \ell(\tilde{\beta})+ \Ex \ell(\b*)\right].
\end{equation}
The first term in \eqref{dec_emp1} can be handled using the regularity of the function $\Ex \ell (\cdot)$ given in assumption \eqref{strong_c}, while to work with the latter we apply assumptions \eqref{ZM} or \eqref{UJ}. 

As previously, we work with the error decomposition \eqref{error_decom} and start with bounding probability of $\{T \notin \mathcal{J}\}.$ Take $a_1 \in (0,1)$ and define $\Ms = \frac{4 \lambda}{(1-a_1) c_1 \cc}\:.$ From \eqref{assump} we have 
$\Ms \leq \bmin.$ Consider the following event
\begin{equation}
\label{omega}
\Omega =\left\{ Z(\Ms) \leq \frac{a_1 \lambda \Ms}{2}
\right\}\:,
\end{equation}
which has probability  not less than 
\begin{equation}
\label{prob}
1- \exp\left(- \frac{a_1^2K_2 \lambda^2}{4 K_1^2  L^2}
\right).
\end{equation}
Indeed, if we take $z= \frac{a_1 \lambda }{2K_1L \sqrt{ \log(2p)}}$ and $r=r^*$ in \eqref{ZM}, then using again \eqref{assump} we obtain
$
z \geq 1
$
and \eqref{prob} follows.

Similarly to the proof of \citet[Theorem 6.4]{BuhlmannGeer11}, we can show that the event 
 \eqref{omega} implies that 
\begin{equation}
\label{l1}
|\hbeta - \b*|_1 \leq \Ms /2. 
\end{equation}
From \eqref{l1} and \eqref{assump} we can obtain separability of the Lasso. Namely,  for each
$j \in T, k \notin T$ that 
$$
|\hbeta _j| \geq |\b* _j|-|\hbeta _j - \b* _j| \geq \bmin - |\hbeta  - \b* |_1 
\geq \Ms/2 \geq |\hbeta  - \b* |_1 \geq |\hbeta _k - \b* _k| = |\hbeta _k|.
$$

Now we consider probability that GIC chooses a submodel of $T$ in the second step of the SS algorithm. Using the definition \eqref{calE}
we obtain
$$
\{T\in {\cal J},\hat{T} \subset T\} \subset \bigcup_{k=1}^{t-1} {\cal E}_k(k\lambda^2/2).
$$
Fix $k$ and $J$ such that $ J\subset T, |T\setminus J|=k.$ We take arbitrary $\beta_T \in \partial B_{2,T}\left( \sqrt{\delta_k}\right)$ and recall that $\mathring{\ell} (\beta)=\ell(\beta) - \ell(\mathring{\beta}).$ It is clear that $\dt \geq \delta_k,$ so we can use \eqref{strong_c} for $\beta_T$ and obtain
\begin{multline*}
\ell(\beta_T) - \ell(\b*) = \Ex \ml(\beta_T)  +[\ml(\beta_T) -\Ex \ml (\beta_T)]
\\ \geq  \frac{c_2}{2} (\b* - \beta_T)^T X_T^TX_T(\b* - \beta_T) - U_T\left(\sqrt{\delta_k}
\right)
\geq  \frac{c_2}{2}\delta_k - U_T\left(\sqrt{\delta_k}\right).
\end{multline*}
Since the last expression does not depend on $\beta_T$ we have
\begin{equation*}
\min_{\beta_T \in \partial B_{2,T}\left( \sqrt{\delta_k}\right)} \ml (\beta_T) \geq 
\frac{c_2}{2}\delta_k - U_T\left(\sqrt{\delta_k}\right).
\end{equation*}
Proceeding as in the proof of Lemma \ref{aux_lemma} we obtain
$$
{\cal E}_k (k\lambda^2/2) \subset \left\{ k\lambda^2/2 \geq \frac{c_2}{2}\delta_k - U_T\left(\sqrt{\delta_k}\right) \right\}.
$$
If we take $b_k = \frac{k \lambda^2 }{c_2 \delta_k}\:,$ then 
$$
{\cal E}_k (k\lambda^2/2) \subset \left\{ U_T\left(\sqrt{\delta_k}\right) \geq \frac{c_2\delta_k}{2} (1-b_k)  \right\}.
$$
From \eqref{assump} we have (as in the proof of Lemma \ref{lemSub2})
\begin{equation}
\label{ssub2}
a_2b_k \leq (1-b_k)^2/4.
\end{equation}
Take
$$z= \frac{(1-b_k)c_2 \sqrt{\delta_k}}{2K_3 L \sqrt{t}}\:,$$
which is not smaller than one by \eqref{ssub2} and \eqref{assump}.
Thus,  using \eqref{UJ} with $r=\sqrt{\delta_k}$ and \eqref{ssub2} we have
$$
P\left(U_T\left(\sqrt{\delta_k}\right) \geq \frac{c_2\delta_k}{2} (1-b_k) 
\right) \leq \exp\left(- \frac{K_4 a_2k  c_2  \lambda^2}{K_3 ^2 L^2}
\right).
$$
Therefore,
\begin{eqnarray*}
&\,&P(T\in {\cal J},\hat{T} \subset T)\leq
\sum_{k=1}^{t-1}\exp\left(-k       \frac{K_4a_2 c_2  \lambda^2}{K_3^2L^2}        \right)\\
&\leq& \left[\exp\left(  \frac{K_4a_2 c_2  \lambda^2}{K_3^2L^2}   \right) -1\right]^{-1}\leq  \frac{\exp(K_4t)}{\exp(K_4t)-1} \exp\left(-  \frac{K_4a_2 c_2  \lambda^2}{K_3^2L^2}   \right),
\end{eqnarray*}
because $\lambda^2\geq \frac{K_3^2tL^2}{a_2c_2}$ by \eqref{assump1}.

Next, we consider choosing a supermodel of $T$ by GIC. 
Fix the set $J\supset T, r(X_J)=|J|\leq \bar{t} $ and denote $m=|J\setminus T|.$ Therefore, we have
\begin{eqnarray*}
&\,&\{-m \lambda^2/2 \geq  \ell(\hbj)- \ell(\hat{\beta}_T)  \} \subset \{-m \lambda^2/2 \geq  \ml(\hbj)  \}\\
&=& \{-m \lambda^2/2 \geq \Ex \ml (\hbj) + [\ml (\hbj) -\Ex \ml (\hbj)]  \}\\
&\subset& \{-m \lambda^2/2 \geq \ml (\hbj) -\Ex \ml (\hbj)  \}=:D_J,
\end{eqnarray*}
because $\ell(\hat{\beta}_T) \leq \ell(\b*)$ and $\Ex \ell(\hbj) \geq \Ex \ell(\b*).$ We will prove that 
$D_J$ is contained in 
\begin{equation}
\label{UJ1}
\{U_J(r) >m \lambda^2/2\}
\end{equation}
for $r=2 \lambda \sqrt{\frac{m}{c_2}}\:.$ 
The event $D_J$ can be decomposed as
\begin{equation}
\label{DJ}
D_J= \{D_J \cap [\hbj \in B_{2,J}(r)]\} \cup \{D_J \cap [\hbj \notin B_{2,J}(r)]\}.
\end{equation}
It is clear that the first event on the right-hand side of \eqref{DJ} is contained 
in \eqref{UJ1} and the second one is contained in $\{\hbj \notin B_{2,J}(r)\}.$ Now we prove that 
$\{\hbj \notin B_{2,J}(r)\}$ is also contained in \eqref{UJ1}. Our argumentation is similar to the proof of Lemma \ref{lemSup}, but here we take $ u=\frac{r}{r+d_J}. $
Therefore, we obtain 
\begin{equation}
\label{notball1}
\Ex \ml(\tbj) = \ml (\tbj) - \ml (\tbj) +\Ex \ml(\tbj) \leq - \ml (\tbj) +\Ex \ml(\tbj) \leq U_J(r).
\end{equation}
Moreover,  the following bound
\begin{equation}
\label{notball2}
\Ex \ell(\tbj) - \Ex \ell(\b*) \geq \frac{c_2}{2} \tilde{d}_J ^2, 
\end{equation}
 is implied by \eqref{strong_c}, because
 $\tbj \in B_{2,J}(r) \subset B_J.$ This inclusion comes from \eqref{assump}, because
$$
\dt \geq \frac{4 \lambda^2 (\bar{t}-t)}{c_2} \geq \frac{4 \lambda^2 m}{c_2}.
$$
Taking \eqref{notball1} and \eqref{notball2} we have
$$
\{\hbj \notin B_{2,J}(r)\} = \{d_J^2>r^2\}= \left\{\tilde{d}_J^2>\frac{r^2}{4}\right\} \subset
\left\{U_J(r) >\frac{c_2 r^2}{8}\right\}
$$
and $\frac{c_2 r^2}{8}=m \lambda^2 /2$ by the choice of $r.$ Therefore, we have just proved that 
$D_J$ is contained in \eqref{UJ1}. To finish the proof we need sharp enough upper bound of \eqref{UJ1}. 
It can be obtained using \eqref{UJ} with  $z=\frac{ \lambda \sqrt{mc_2}}{4K_3 L \sqrt{|J|}}$ that gives us 
$$
P(D_J) \leq P(U_J(r) >m \lambda^2/2) \leq \exp\left(- \frac{K_4 m \lambda ^2 c_2}{16 K_3^2 L^2}
\right).
$$
Notice that $z\geq 1, $ because using \eqref{assump}
$$
\lambda^2 \geq \frac{16 K_3^2 L^2(t+1) }{c_2} \geq \frac{16 K_3^2 L^2 |J|}{c_2 m}\:.
$$
Thus, we obtain the following bound on probability of choosing a supermodel
\begin{eqnarray*}
&\,&P(T\in {\cal J},\hat{T} \supset T) \leq \sum_{m=1}^{p-t} {{p-t}\choose{m}}\exp\left(
-m\frac{ K_4 \lambda ^2 c_2}{ 16 K_3^2L^2}\right)\\
&\leq& \sum_{m=1}^{p-t} \frac{1}{m!} \exp \left[-m\left(\frac{K_4 \lambda ^2 c_2}{16 K_3^2 L^2} -\log (p-t)\right)\right]\\
&\leq& (\log 2)^{-1} \exp\left(-\frac{ K_4 \lambda ^2 c_2}{32 K_3^2 L^2}
\right),
\end{eqnarray*}
where we use the fact that $\frac{K_4 \lambda ^2 c_2}{ 16K_3^2 L^2} \geq 2\log (p-t)$
from \eqref{assump} and two inequalities $ {{p-t}\choose{m}} \leq \frac{(p-t)^m}{m!}$ and 
$\exp(b)-1 \leq (\log 2)^{-1} b$ for $b \in (0,\log 2).$ \qed

\bibliographystyle{apalike}
%\bibliography{bibliography}

\bibliography{SS_arxiv}

\begin{thebibliography}{}

\bibitem[Bartlett et~al., 2006]{bja:06}
Bartlett, P.~L., Jordan, M.~I., and McAuliffe, J.~D. (2006).
\newblock Convexity, classification and risk bounds.
\newblock {\em Journal of the American Statistical Association}, 101:138--156.

\bibitem[Bickel et~al., 2009]{BickelEtAl09}
Bickel, P., Ritov, Y., and Tsybakov, A. (2009).
\newblock Simultaneous analysis of {L}asso and {D}antzig selector.
\newblock {\em Annals of Statistics}, 37:1705--1732.

\bibitem[Breheny and Huang, 2011]{ncvreg}
Breheny, P. and Huang, J. (2011).
\newblock Coordinate descent algorithms for nonconvex penalized regression,
  with applications to biological feature selection.
\newblock {\em Annals of Applied Statistics}, 5:232--253.

\bibitem[B\"uhlmann and van~de Geer, 2011]{BuhlmannGeer11}
B\"uhlmann, P. and van~de Geer, S. (2011).
\newblock {\em Statistics for High-dimensional Data}.
\newblock Springer, New York.

\bibitem[Efron et~al., 2004]{EfronEtAl04}
Efron, B., Hastie, T., Johnstone, I., and Tibshirani, R. (2004).
\newblock Least angle regression.
\newblock {\em Annals of Statistics}, 32:407--499.

\bibitem[Fan and Li, 2001]{FanLi01}
Fan, J. and Li, R. (2001).
\newblock Variable selection via nonconcave penalized likelihood and its oracle
  properties.
\newblock {\em Journal of the American Statistical Association}, 96:1348--1360.

\bibitem[Fan et~al., 2014]{FanEtAl14}
Fan, J., Xue, L., and Zou, H. (2014).
\newblock Strong oracle optimality of folded concave penalized estimation.
\newblock {\em Annals of Statistics}, 42:819--849.

\bibitem[Fan and Tang, 2013]{FanTang2013}
Fan, Y. and Tang, C.~Y. (2013).
\newblock Tuning parameter selection in high dimensional penalized likelihood.
\newblock {\em Journal of the Royal Statistical Society. Series B (Statistical
  Methodology)}, 75:531--552.

\bibitem[Foster and George, 1994]{FosterGeorge1994}
Foster, D. and George, E. (1994).
\newblock The risk inflation criterion for multiple regression.
\newblock {\em Annals of Statistics}, 22:1947--1975.

\bibitem[Friedman et~al., 2010]{FriedmanEtAl10}
Friedman, J., Hastie, T., and Tibshirani, R. (2010).
\newblock Regularization paths for generalized linear models via coordinate
  descent.
\newblock {\em Journal of Statistical Software}, 33:1--22.

\bibitem[Gravier et~al., 2010]{GravierEtAl10}
Gravier, E., Pierron, G., Vincent-Salomon, A., Gruel, N., Raynal, V.,
  Savignoni, A., De~Rycke, Y., Pierga, J., Lucchesi, C., Reyal, F., et~al.
  (2010).
\newblock {Prognostic DNA signature for T1T2 node-negative breast cancer
  patients}.
\newblock {\em Genes, Chromosomes and Cancer}, 49(12):1125--1125.

\bibitem[Hannum et~al., 2013]{HannumEtAl13}
Hannum, G., Guinney, J., Zhao, L., Zhang, L., Hughes, G., Sadda, S., Klotzle,
  B., Bibikova, M., Fan, J., Gao, Y., et~al. (2013).
\newblock Genome-wide methylation profiles reveal quantitative views of human
  aging rates.
\newblock {\em Molecular cell}, 49(2):359--367.

\bibitem[Hsu et~al., 2012]{HsuEtAl12}
Hsu, D., Kakade, S., and Zhang, T. (2012).
\newblock A tail inequality for quadratic forms of subgaussian random vectors.
\newblock {\em Electronic Communications in Probability}, 17:1--6.

\bibitem[Huang and Zhang, 2012]{HuangZhang12}
Huang, J. and Zhang, C. (2012).
\newblock Estimation and selection via absolute penalized convex minimization
  and its multistage adaptive applications.
\newblock {\em Journal of Machine Learning Research}, 13:1839--1864.

\bibitem[Hui et~al., 2015]{HuiWartonFoster2015}
Hui, F. K.~C., Warton, D.~I., and Foster, S.~D. (2015).
\newblock Tuning parameter selection for the adaptive lasso using eric.
\newblock {\em Journal of the American Statistical Association}, 110:262--269.

\bibitem[Katayama and Imori, 2014]{KATAYAMA2014138}
Katayama, S. and Imori, S. (2014).
\newblock Lasso penalized model selection criteria for high-dimensional
  multivariate linear regression analysis.
\newblock {\em Journal of Multivariate Analysis}, 132:138 -- 150.

\bibitem[Kim and Jeon, 2016]{KimJeon2016}
Kim, Y. and Jeon, J.-J. (2016).
\newblock Consistent model selection criteria for quadratically supported
  risks.
\newblock {\em Ann. Statist.}, 44:2467--2496.

\bibitem[Kim et~al., 2012]{KimKwonChoi2012}
Kim, Y., Kwon, S., and Choi, H. (2012).
\newblock Consistent model selection criteria on high dimensions.
\newblock {\em J. Mach. Learn. Res.}, 13:1037--1057.

\bibitem[Ledoux and Talagrand, 1991]{ledtal:91}
Ledoux, M. and Talagrand, M. (1991).
\newblock {\em Probability in Banach Spaces: Isoperimetry and Processes}.
\newblock Springer, Berlin.

\bibitem[Massart, 2000]{Massart2000}
Massart, P. (2000).
\newblock About the constants in {T}alagrand concentration inequalities for
  empirical processes.
\newblock {\em The Annals of Probability}, 28:863--884.

\bibitem[Mazumder et~al., 2011]{MazumderEtAl11}
Mazumder, R., Friedman, J., and Hastie, T. (2011).
\newblock Sparse{N}et: Coordinate descent with nonconvex penalties.
\newblock {\em Journal of the American Statistical Association},
  106:1125--1138.

\bibitem[Meinshausen and B\"uhlmann, 2006]{MeinshausenBuhlmann06}
Meinshausen, N. and B\"uhlmann, P. (2006).
\newblock High dimensional graphs and variable selection with the {L}asso.
\newblock {\em Annals of Statistics}, 34:1436--1462.

\bibitem[Pokarowski and Mielniczuk, 2015]{PokarowskiMielniczuk15}
Pokarowski, P. and Mielniczuk, J. (2015).
\newblock Combined $\ell_0$ and $\ell_1$ penalized least squares for linear
  model selection.
\newblock {\em Journal of Machine Learning Research}, 16:961--992.

\bibitem[Shao, 1997]{shao97}
Shao, J. (1997).
\newblock An asymptotic theory for linear model selection.
\newblock {\em Statistica Sinica}, 7:221--242.

\bibitem[Shen et~al., 2012]{ShenEtAl12}
Shen, X., Pan, W., and Zhu, Y. (2012).
\newblock Likelihood-based selection and sharp parameter estimation.
\newblock {\em Journal of the American Statistical Association}, 107:223--232.

\bibitem[Shen et~al., 2013]{ShenEtAl13}
Shen, X., Pan, W., Zhu, Y., and Zhou, H. (2013).
\newblock On constrained and regularized high-dimensional regression.
\newblock {\em Annals of the Institute of Statistical Mathematics},
  65:807--832.

\bibitem[Tibshirani, 2011]{Tibshirani11}
Tibshirani, R. (2011).
\newblock Regression shrinkage and selection via the {L}asso: a retrospective.
\newblock {\em Journal of the Royal Statistical Society Series B}, 73:273--282.

\bibitem[van~de Geer, 2002]{geer02}
van~de Geer, S. (2002).
\newblock M-estimation using penalties or sieves.
\newblock {\em Journal of Statistical Planning and Inference}, 108:55--69.

\bibitem[van~de Geer and B\"uhlmann, 2009]{GeerBuhlmann09}
van~de Geer, S. and B\"uhlmann, P. (2009).
\newblock On the conditions used to prove oracle results for the {L}asso.
\newblock {\em Electronic Journal of Statistics}, 3:1360--1392.

\bibitem[van~der Vaart and Wellner, 1996]{vw:96}
van~der Vaart, A.~W. and Wellner, J.~A. (1996).
\newblock {\em {W}eak Convergence and Empirical Processes: With Applications to
  Statistics}.
\newblock Springer Verlag, New York.

\bibitem[Wallace, 1959]{Wallace59}
Wallace, D. (1959).
\newblock Bounds on normal approximations to {S}tudent's and the chi-square
  distributions.
\newblock {\em Annals of Mathematical Statistics}, 30:1121--1130.

\bibitem[Wang et~al., 2013]{WangEtAl13}
Wang, L., Kim, Y., and Li, R. (2013).
\newblock Calibrating non-convex penalized regression in ultra-high dimension.
\newblock {\em Annals of Statistics}, 41:2505--2536.

\bibitem[Wang et~al., 2014]{WangEtAl14}
Wang, Z., Liu, H., and Zhang, T. (2014).
\newblock Optimal computational and statistical rates of convergence for sparse
  nonconvex learning problems.
\newblock {\em Annals of Statistics}, 42:2164--2201.

\bibitem[Ye and Zhang, 2010]{YeZhang10}
Ye, F. and Zhang, C. (2010).
\newblock Rate minimaxity of the {L}asso and {D}antzig {S}elector for the $l_q$
  loss in $l_r$ balls.
\newblock {\em Journal of Machine Learning Research}, 11:3519--3540.

\bibitem[Zhang, 2010a]{ZhangCH10}
Zhang, C. (2010a).
\newblock Nearly unbiased variable selection under minimax concave penalty.
\newblock {\em Annals of Statistics}, 38:894--942.

\bibitem[Zhang and Zhang, 2012]{ZhangZhang12}
Zhang, C. and Zhang, T. (2012).
\newblock A general theory of concave regularization for high-dimensional
  sparse estimation problems.
\newblock {\em Statistical Science}, 27:576--593.

\bibitem[Zhang, 2010b]{ZhangT10}
Zhang, T. (2010b).
\newblock Analysis of multi-stage convex relaxation for sparse regularization.
\newblock {\em Journal of Machine Learning Research}, 11:1081--1107.

\bibitem[Zhang et~al., 2016]{ZhangSVM2016}
Zhang, X., Wu, Y., Wang, L., and Li, R. (2016).
\newblock A consistent information criterion for support vector machines in
  diverging model spaces.
\newblock {\em J. Mach. Learn. Res.}, 17:466--491.

\bibitem[Zhao and Yu, 2006]{ZhaoYu06}
Zhao, P. and Yu, B. (2006).
\newblock On model selection consistency of {L}asso.
\newblock {\em Journal of Machine Learning Research}, 7:2541--2563.

\bibitem[Zhou, 2009]{Zhou09}
Zhou, S. (2009).
\newblock Thresholding procedures for high dimensional variable selection and
  statistical estimation.
\newblock In {\em NIPS}, pages 2304--2312.

\end{thebibliography}

\end{document}